\documentclass[11pt,reqno]{amsart}
\usepackage{fullpage}
\usepackage{mathrsfs,amssymb,graphicx,verbatim,amsmath,amsfonts}
\usepackage{paralist}
\usepackage[breaklinks,pdfstartview=FitH]{hyperref}
\usepackage{upgreek}
\usepackage{tikz}
\usepackage[mathscr]{euscript}
\usepackage{relsize}
 \usepackage{amsaddr}

\makeatletter
\@namedef{subjclassname@2020}{%
  \textup{2020} Mathematics Subject Classification}
\makeatother

\usepackage{dutchcal}

\newcommand{\MM}{\mathcal{M}}

\hypersetup{ocgcolorlinks=true,allcolors=testc}
\hypersetup{
     colorlinks   = true,
     citecolor    = black
}
\hypersetup{linkcolor=black}
\hypersetup{urlcolor=black}

\renewcommand{\le}{\leqslant}

\renewcommand{\leq}{\leqslant}
\renewcommand{\geq}{\geqslant}
\renewcommand{\setminus}{\smallsetminus}
\renewcommand{\gamma}{\upgamma}
\allowdisplaybreaks

\usepackage{esint}
\usepackage[autostyle]{csquotes}

\renewcommand{\pi}{\uppi}
\newcommand{\NN}{\mathcal{N}}

\newcommand{\e}{\varepsilon}
\newcommand{\R}{\mathbb R}

\newtheorem{theorem}{Theorem}
\newtheorem{lemma}[theorem]{Lemma}
\newtheorem{proposition}[theorem]{Proposition}

\newtheorem{definition}[theorem]{Definition}

\theoremstyle{remark}
\newtheorem{remark}[theorem]{Remark}

\newcounter{quest}
\newtheorem{question}[quest]{Question}

\renewcommand{\tau}{\uptau}

\renewcommand{\xi}{\upxi}
\renewcommand{\rho}{\uprho}

\renewcommand{\subset}{\subseteq}
\newcommand{\C}{\mathbb C}

\newcommand{\E}{\mathbb{ E}}

\newcommand{\N}{\mathbb N}

\newcommand{\eqdef}{\stackrel{\mathrm{def}}{=}}

\renewcommand{\theta}{\uptheta}
\renewcommand{\lambda}{\uplambda}

\renewcommand{\gamma}{\upgamma}
\renewcommand{\beta}{\upbeta}
\renewcommand{\alpha}{\upalpha}
\renewcommand{\kappa}{\upkappa}
\renewcommand{\psi}{\uppsi}
\renewcommand{\rho}{\uprho}
\renewcommand{\delta}{\updelta}
\renewcommand{\pi}{\uppi}
\renewcommand{\omega}{\upomega}
\renewcommand{\sigma}{\upsigma}

\renewcommand{\eta}{\upeta}

\renewcommand{\kappa}{\upkappa}
\renewcommand{\mu}{\upmu}
\renewcommand{\nu}{\upnu}
\renewcommand{\pi}{\uppi}
\renewcommand{\zeta}{\upzeta}

\newcommand{\mb}{\mathbb}
\newcommand*\diff{\mathop{}\!\mathrm{d}}

\newcommand{\ms}{\mathscr}
\newcommand{\msf}{\mathsf}

\begin{document}

\title{Talagrand's influence inequality revisited}
\author{Dario Cordero-Erausquin and Alexandros Eskenazis}
\address{Institut de Math\'ematiques de Jussieu, Sorbonne Universit\'e, Paris, 75252, France}
\thanks{{\it E-mail addresses:} \href{mailto:dario.cordero@imj-prg.fr}{\nolinkurl{dario.cordero@imj-prg.fr}}, $\{$\href{alexandros.eskenazis@imj-prg.fr}{\nolinkurl{alexandros.eskenazis@imj-prg.fr}}, \href{ae466@cam.ac.uk}{\nolinkurl{ae466@cam.ac.uk}}$\}$.
}
\thanks{A.~E.~was supported by a postdoctoral fellowship of the Fondation Sciences Math\'ematiques de Paris.}


\vspace{-0.25in}

\begin{abstract} 
Let $\ms{C}_n=\{-1,1\}^n$ be the discrete hypercube equipped with the uniform probability measure $\sigma_n$. Talagrand's influence inequality (1994), also known as the $L_1-L_2$ inequality, asserts that there exists $C\in(0,\infty)$ such that for every $n\in\N$, every function $f:\ms{C}_n\to\C$ satisfies
$$\mathrm{Var}_{\sigma_n}(f) \leq C \sum_{i=1}^n \frac{\|\partial_if\|_{L_2(\sigma_n)}^2}{1+\log\big(\|\partial_if\|_{L_2(\sigma_n)}/\|\partial_i f\|_{L_1(\sigma_n)}\big)}.$$
In this work, we undertake a systematic investigation of this and related inequalities via harmonic analytic and stochastic techniques and derive applications to metric embeddings. We prove that Talagrand's inequality extends, up to an additional doubly logarithmic factor, to Banach space-valued functions under the necessary assumption that the target space has Rademacher type 2 and that this doubly logarithmic term can be omitted if the target space admits an equivalent 2-uniformly smooth norm. These are the first vector-valued extensions of Talagrand's influence inequality. Moreover, our proof implies vector-valued versions of a general family of $L_1-L_p$ inequalities, each refining the dimension independent $L_p$-Poincar\'e inequality on $(\ms{C}_n,\sigma_n)$. We also obtain a joint strengthening of results of Bakry--Meyer (1982) and Naor--Schechtman (2002) on the action of negative powers of the hypercube Laplacian on functions $f:\ms{C}_n\to E$, whose target space $(E,\|\cdot\|_E)$ has nontrivial Rademacher type via a new vector-valued version of Meyer's multiplier theorem (1984). Inspired by Talagrand's influence inequality, we introduce a new metric invariant called Talagrand type and estimate it for Banach spaces with prescribed Rademacher or martingale type, Gromov hyperbolic groups and simply connected Riemannian manifolds of pinched negative curvature. Finally, we prove that Talagrand type is an obstruction to the bi-Lipschitz embeddability of nonlinear quotients of the hypercube $\ms{C}_n$ equipped with the Hamming metric, thus deriving new nonembeddability results for these finite metrics. Our proofs make use of Banach space-valued It\^o calculus, Riesz transform inequalities, Littlewood--Paley--Stein theory and hypercontractivity.
\end{abstract}
\maketitle

\vspace{-0.4in}

\setcounter{tocdepth}{1}
\tableofcontents

%
%

\section{Introduction}

Let $\ms{C}_n=\{-1,1\}^n$ be the discrete hypercube equipped with the uniform probability measure $\sigma_n$. If $(E,\|\cdot\|_E)$ is a complex Banach space, we will denote the vector-valued $L_p(\sigma_n)$-norm of a function $f:\ms{C}_n\to E$ by
\begin{equation}
\forall \ p\in[1,\infty), \qquad\|f\|_{L_p(\sigma_n;E)} \eqdef \Big( \int_{\ms{C}_n} \|f(\e)\|_E^p\diff\sigma_n(\e) \Big)^{1/p}
\end{equation}
and $\|f\|_{L_\infty(\sigma_n;E)} \eqdef \max_{\e\in\ms{C}_n} \|f(\e)\|_E$. When $E=\C$, we will abbreviate $\|f\|_{L_p(\sigma_n;\C)}$ simply as $\|f\|_{L_p(\sigma_n)}$. We will also denote by $\mb{E}_{\sigma_n} f$ the expectation of $f$ with respect to $\sigma_n$. The $i$-th partial derivative of a function $f:\ms{C}_n\to E$ is given by
\begin{equation} \label{eq:partialderivative}
\forall \ \e\in\ms{C}_n, \qquad\partial_if(\e) = \frac{f(\e)-f(\e_1,\ldots,\e_{i-1},-\e_i,\e_{i+1},\ldots,\e_n)}{2}.
\end{equation}
The discrete Poincar\'e inequality asserts that every function $f:\ms{C}_n\to\C$ satisfies 
\begin{equation} \label{eq:poincare}
\big\| f-\mb{E}_{\sigma_n}f\big\|_{L_2(\sigma_n)}^2 \leq \sum_{i=1}^n \|\partial_if\|_{L_2(\sigma_n)}^2.
\end{equation}
Extensions and refinements of~\eqref{eq:poincare} have been a central object of study in the probability and analysis literature for decades. A natural problem, first raised by Enflo~\cite{Enf78}, is to understand for which target spaces $E$, every function $f:\ms{C}_n\to E$ satisfies~\eqref{eq:poincare} up to a universal multiplicative factor depending only on the geometry of $E$ but not $n$ or the choice of $f$. Recall that a Banach space $(E,\|\cdot\|_E)$ has Rademacher type $s$ with constant $T\in(0,\infty)$ if for every $n\in\N$ and $x_1,\ldots,x_n\in E$,
\begin{equation}
\int_{\ms{C}_n} \Big\|\sum_{i=1}^n \e_i x_i\Big\|_E^s \diff\sigma_n(\e) \leq T^s \sum_{i=1}^n \|x_i\|_E^s.
\end{equation}
It is evident that if a Banach space $E$ is such that every function $f:\ms{C}_n\to E$ satisfies
\begin{equation} \label{eq:vectorpoincare}
\big\| f-\mb{E}_{\sigma_n}f\big\|^2_{L_2(\sigma_n;E)} \leq C^2 \sum_{i=1}^n \|\partial_if\|_{L_2(\sigma_n;E)}^2,
\end{equation}
then $E$ has Rademacher type 2 with constant $C$, since this condition coincides with~\eqref{eq:vectorpoincare} for functions of the form $f(\e)=\sum_{i=1}^n\e_i x_i$, where $x_1,\ldots,x_n\in E$. The reverse implication, i.e.~the fact that Rademacher type 2 implies the vector-valued Poincar\'e inequality~\eqref{eq:vectorpoincare}, was proven in the recent breakthrough~\cite{IVV20} of Ivanisvili, van Handel and Volberg.

In a different direction, an important refinement of the scalar-valued discrete Poincar\'e inequality~\eqref{eq:poincare} was obtained by Talagrand in the celebrated work~\cite{Tal94}. Talagrand's influence inequality, also known as the $L_1-L_2$ inequality, asserts that there exists a universal constant $C\in(0,\infty)$ such that for every $n\in\N$, every function $f:\ms{C}_n\to\C$ satisfies
\begin{equation} \label{eq:talagrand}
\big\| f-\mb{E}_{\sigma_n}f\big\|_{L_2(\sigma_n)}^2 \leq C \sum_{i=1}^n \frac{\|\partial_if\|_{L_2(\sigma_n)}^2}{1+\log\big(\|\partial_if\|_{L_2(\sigma_n)}/\|\partial_i f\|_{L_1(\sigma_n)}\big)}.
\end{equation}
Observe that~\eqref{eq:talagrand} is a strengthening of the discrete Poincar\'e inequality~\eqref{eq:poincare} up to the value of the universal constant $C$, which becomes substantial for functions satisfying $\|\partial_if\|_{L_2(\sigma_n)} >\!\!\!>\|\partial_if\|_{L_1(\sigma_n)}$. Since its conception, Talagrand's inequality has played a major role in Boolean analysis~\cite{KKL88, FK96, Ros06, FS07, O'Do14}, percolation~\cite{Rus82, BKS03, BR08, Cha14, GS15} and geometric functional analysis~\cite{PVZ17, PV18, Tik18, PVT19}. In particular, applying~\eqref{eq:talagrand} to a Boolean function $f:\ms{C}_n\to\{0,1\}$, one readily recovers the celebrated theorem of Kahn, Kalai and Linial \cite{KKL88}, quantifying the fact that in any (essentially) unbiased voting scheme, there exists a voter with disproportionately large influence over the outcome of the vote. We refer to the above references and~\cite{CL12, Led19} for further bibliographical information on Talagrand's inequality. 

The main purpose of the present paper is to investigate vector-valued versions of Talagrand's inequality~\eqref{eq:talagrand} and other refinements and extensions of~\eqref{eq:poincare}. These new vector-valued inequalities motivate the definition of a new bi-Lipschitz invariant for metric spaces called {\em Talagrand type} (Definition \ref{def:talagrand}), which captures new KKL-type phenomena in embedding theory (see Theorem \ref{thm:embed} and the ensuing discussion). We shall now present a summary of these results, which rely on a range of stochastic and harmonic analytic tools such as Banach space-valued It\^o calculus, Riesz transforms and Littlewood--Paley--Stein theory, along with standard uses of hypercontractivity.

\medskip
\noindent{\bf Asymptotic notation.}  In what follows we use the convention that for $a,b\in[0,\infty]$ the notation $a\gtrsim b$ (respectively $a\lesssim b$) means that there exists a universal constant $c\in(0,\infty)$ such that $a\geq cb$ (respectively $a\leq cb$). Moreover, $a\asymp b$ stands for $(a\lesssim b)\wedge(a\gtrsim b)$. The notations $\lesssim_\xi, \gtrsim_\chi$ and $\asymp_\psi$ mean that the implicit constant $c$ depends on $\xi, \chi$ and $\psi$ respectively.


\subsection{Vector-valued influence inequalities} In view of Enflo's problem~\cite{Enf78} and its recent solution in~\cite{IVV20}, it would be most natural to try and understand for which Banach spaces $(E,\|\cdot\|_E)$ there exists a constant $C=C(E)\in(0,\infty)$ such that for every $n\in\N$, every function $f:\ms{C}_n\to E$ satisfies
\begin{equation} \label{eq:vectortalagrand}
\big\| f-\mb{E}_{\sigma_n}f\big\|_{L_2(\sigma_n;E)}^2 \leq C \sum_{i=1}^n \frac{\|\partial_if\|_{L_2(\sigma_n;E)}^2}{1+\log\big(\|\partial_if\|_{L_2(\sigma_n;E)}/\|\partial_i f\|_{L_1(\sigma_n;E)}\big)}.
\end{equation}
Evidently, as~\eqref{eq:vectortalagrand} is a strengthening of~\eqref{eq:vectorpoincare}, if a space $(E,\|\cdot\|_E)$ satisfies~\eqref{eq:vectortalagrand} then $E$ has Rademacher type 2. Conversely, we shall prove the following theorem.

\begin{theorem} [Vector-valued influence inequality for spaces with Rademacher type 2]  \label{thm:useivv}
Let $(E,\|\cdot\|_E)$ be a Banach space with Rademacher type 2. Then, there exists $C=C(E)\in(0,\infty)$ such that for every $\e\in(0,1)$ and $n\in\N$, every function $f:\ms{C}_n\to E$ satisfies
\begin{equation} \label{eq:thmuseivv}
\big\| f-\mb{E}_{\sigma_n}f\big\|_{L_2(\sigma_n;E)}^2 \leq \frac{C}{\e} \sum_{i=1}^n \frac{\|\partial_if\|_{L_2(\sigma_n;E)}^2}{1+\log^{1-\e}\big(\|\partial_if\|_{L_2(\sigma_n;E)}/\|\partial_i f\|_{L_1(\sigma_n;E)}\big)}.
\end{equation}
In particular, if $\sigma(f) \eqdef \max_{i\in\{1,\ldots,n\}}\log\log\big(e+\|\partial_if\|_{L_2(\sigma_n;E)}/\|\partial_i f\|_{L_1(\sigma_n;E)}\big)$, then
\begin{equation} \label{eq:thmuseivvloglog}
\big\| f-\mb{E}_{\sigma_n}f\big\|_{L_2(\sigma_n;E)}^2 \leq C\sigma(f) \sum_{i=1}^n \frac{\|\partial_if\|_{L_2(\sigma_n;E)}^2}{1+\log\big(\|\partial_if\|_{L_2(\sigma_n;E)}/\|\partial_i f\|_{L_1(\sigma_n;E)}\big)}.
\end{equation}
\end{theorem}
The proof of Theorem~\ref{thm:useivv} builds upon a novel idea exploited in~\cite{IVV20}, which in turn is reminiscent of a trick due to Maurey~\cite{Pis86}.
It remains unclear whether one can deduce from this idea a vector-valued extension of Talagrand's inequality~\eqref{eq:talagrand} for spaces of Rademacher type 2 and whether the doubly logarithmic error term $\sigma(f)$ on the right hand side of \eqref{eq:thmuseivvloglog} is needed.
Let us mention that, even in the scalar-valued case,  the argument of Maurey or the one of Ivanisvili, van Handel and Volberg are slightly different than standard semigroup approaches to functional inequalities, in particular to the semigroup proof of~\eqref{eq:talagrand}
 from~\cite{CL12}. 
 On the other hand,   
 we will see that a slightly stronger condition on the Banach space allows for different approaches, relying on more intricate connections between the space and the semigroup, which will lead to the desired optimal vector-valued $L_1-L_2$ inequality.
 Recall first that a Banach space $(E,\|\cdot\|_E)$ has martingale type s with constant $M\in(0,\infty)$ if for every $n\in\N$, every probability space $(\Omega,\ms{F},\mu)$ and every filtration $\{\ms{F}_i\}_{i=0}^n$ of sub-$\sigma$-algebras of $\ms{F}$, \mbox{every $E$-valued martingale $\{\ms{M}_i:\Omega\to E\}_{i=0}^n$ adapted to $\{\ms{F}_i\}_{i=0}^n$ satisfies}
\begin{equation} \label{eq:mtype}
\big\|\ms{M}_n-\ms{M}_0\big\|_{L_s(\mu;E)}^s \leq M^s \sum_{i=1}^n \big\| \ms{M}_i-\ms{M}_{i-1}\big\|^s_{L_s(\mu;E)}.
\end{equation}
Martingale type, which is a strengthening of Rademacher type, was introduced by Pisier in~\cite{Pis75}, who proved the fundamental fact that for every $s\in(1,2]$, a Banach space $E$ has martingale type $s$ if and only if $E$ admits an equivalent $s$-uniformly smooth norm (see~\cite{Pis75, Pis16} for further information on these important notions).

\begin{theorem}  [Vector-valued influence inequality for spaces with martingale type 2]  \label{thm:useeldan}
Let $(E,\|\cdot\|_E)$ be a Banach space with martingale type 2. Then, there exists $C=C(E)\in(0,\infty)$ such that for every $n\in\N$, every function $f:\ms{C}_n\to E$ satisfies
\begin{equation} \label{eq:thmuseeldan}
\big\| f-\mb{E}_{\sigma_n}f\big\|_{L_2(\sigma_n;E)}^2 \leq C \sum_{i=1}^n \frac{\|\partial_if\|_{L_2(\sigma_n;E)}^2}{1+\log\big(\|\partial_if\|_{L_2(\sigma_n;E)}/\|\partial_i f\|_{L_1(\sigma_n;E)}\big)}.
\end{equation}
\end{theorem}

Theorem~\ref{thm:useeldan} establishes the optimal vector-valued influence inequality for spaces of martingale type 2. We will present two proofs of Theorem~\ref{thm:useeldan}. The first one uses a clever stochastic process on the cube which was recently constructed by Eldan and Gross~\cite{EG19}, while the second relies on Xu's vector-valued Littlewood--Paley--Stein inequalities for superreflexive targets (see~\cite{Xu20}). There exist examples of exotic Banach spaces~\cite{Jam78, PX87} which have Rademacher type 2 yet fail to have martingale type 2, thus Theorem~\ref{thm:useeldan} does not exhaust the list of potential target spaces satisfying~\eqref{eq:vectortalagrand}. Nevertheless, a combination of classical results of Maurey~\cite{Mau74}, Pisier~\cite{Pis75} and Figiel~\cite{Fig76} imply that every Banach lattice of Rademacher type 2 has martingale type 2. 

The influence inequalities of Theorems \ref{thm:useivv} and \ref{thm:useeldan} have analogues for spaces of Rademacher and martingale type $s$ which will be presented in Section \ref{sec:9.1} for the sake of simplicity of exposition.


\subsection{$L_1-L_p$ inequalities} For a function $f:\ms{C}_n\to\C$ denote by
\begin{equation} \label{eq:defnormgrad}
\forall \ p\in[1,\infty), \qquad\big\|\nabla f\big\|_{L_p(\sigma_n)} \eqdef \Big\|\Big( \sum_{i=1}^n\big(\partial_i f\big)^2\Big)^{1/2}\Big\|_{L_p(\sigma_n)}
\end{equation}
the $L_p$ norm of the gradient of $f$. It has already been pointed out that Talagrand's influence inequality~\eqref{eq:talagrand} is a refinement of the discrete Poincar\'e inequality~\eqref{eq:poincare}. It is therefore worth investigating whether similar strengthenings of the $L_p$ discrete Poincar\'e inequality
\begin{equation} \label{eq:lppoincare}
\big\|f-\mb{E}_{\sigma_n}f\big\|_{L_p(\sigma_n)} \leq C_p \big\|\nabla f\big\|_{L_p(\sigma_n)}
\end{equation}
hold true for other values of $p$. The fact that for every $p\in[1,\infty)$ there exists a constant $C_p\in(0,\infty)$ such that~\eqref{eq:lppoincare} holds true for every $n\in\N$ and $f:\ms{C}_n\to\C$ was established by Talagrand in~\cite{Tal93}.

In the vector-valued setting which is of interest here, the most common substitute of~\eqref{eq:defnormgrad} for the norm of the gradient of a function $f:\ms{C}_n\to E$, where $(E,\|\cdot\|_E)$ is a Banach space, is
\begin{eqnarray*}
\forall \ p\in[1,\infty), \quad \big\|\nabla f\big\|_{L_p(\sigma_n;E)} & \eqdef & \Big(\int_{\ms{C}_n} \Big\|\sum_{i=1}^n \delta_i \partial_i f \Big\|_{L_p(\sigma_n;E)}^p \diff\sigma_n(\delta)\Big)^{1/p}\\
& =&  \Big(\int_{\ms{C}_n\times\ms{C}_n}  \Big\|\sum_{i=1}^n \delta_i \partial_i f(\varepsilon) \Big\|_{E}^p \diff\sigma_{2n}(\varepsilon,\delta) \Big)^{1/p}
\end{eqnarray*}
Observe that when $E=\C$, for every $p\in[1,\infty)$, we have $\big\|\nabla f\big\|_{L_p(\sigma_n;\C)} \asymp_p \big\|\nabla f\big\|_{L_p(\sigma_n)}$ by Khintchine's inequality~\cite{Khi23}. With this definition, the vector-valued extension
\begin{equation} \label{eq:pisierineq}
\big\|f-\mb{E}_{\sigma_n}f\big\|_{L_p(\sigma_n;E)} \leq C_p(n) \big\|\nabla f\big\|_{L_p(\sigma_n;E)}
\end{equation}
of~\eqref{eq:lppoincare} is called Pisier's inequality, since Pisier established in~\cite{Pis86} the validity of~\eqref{eq:pisierineq} for every Banach space $E$ and $p\in[1,\infty)$ with $C_p(n)=2e\log n$. Understanding for which Banach spaces $E$ and $p\in[1,\infty)$ the constant $C_p(n)$ in Pisier's inequality could be replaced by a constant $C_p(E)$, independent of the dimension $n$, was a long-standing open problem settled in the recent work~\cite{IVV20}. We will recall below, in~\eqref{eq:cotype-q}, the definition of Rademacher cotype; let us simply say here that a Banach space $E$ has finite cotype  if $E$ does not isomorphically contain the family $\{\ell_\infty^n\}_{n=1}^\infty$ with uniformly bounded distortion (see~\cite{MP76, Pis16}). In~\cite{IVV20}, Ivanisvili, van Handel and Volberg proved that a Banach space $E$ with finite cotype satisfies~\eqref{eq:pisierineq} with $C_p(n)$ replaced by a universal constant $C_p(E)$, thus complementing a result of Talagrand~\cite{Tal93} who proved that if a space does not have finite cotype, then $C_p(n) \asymp_p \log n$.

\begin{theorem} [Vector-valued $L_1-L_p$ inequality for spaces of finite cotype] \label{thm:lptalagrand}
Let $(E,\|\cdot\|_E)$ be a Banach space with finite Rademacher cotype and $p\in(1,\infty)$. Then, there exists $C_p=C_p(E)\in(0,\infty)$ and $\alpha_p=\alpha_p(E)\in\big(0,\tfrac{1}{2}\big]$ such that for every $n\in\N$, every function $f:\ms{C}_n\to E$ satisfies
\begin{equation} \label{eq:thmlptalagrand}
\big\| f-\mb{E}_{\sigma_n}f\big\|_{L_p(\sigma_n;E)} \leq C_p \frac{\|\nabla f\|_{L_p(\sigma_n;E)}}{1+\log^{\alpha_p}\big(\|\nabla f\|_{L_p(\sigma_n;E)}/\|\nabla f\|_{L_1(\sigma_n;E)}\big)}.
\end{equation}
\end{theorem}

The proof of Theorem~\ref{thm:lptalagrand} builds upon the technique of~\cite{IVV20}. A stronger inequality for functions on the Gauss space will be presented in Theorem~\ref{thm:usemp}. This approach seems insufficient to yield the optimal $\alpha_p=\tfrac{1}{2}$ exponent for $E=\C$ and all $p>1$, yet we derive the following result using Lust-Piquard's Riesz transform inequalities~\cite{LP98, BELP08}. 

\begin{theorem} [Scalar-valued $L_1-L_p$ inequality] \label{thm:l1lpscalar}
For every $p\in(1,\infty)$, there exists $C_p\in(0,\infty)$ such that for every $n\in\N$, every function $f:\ms{C}_n\to\C$ satisfies
\begin{equation} \label{eq:thml1lpscalar}
\big\| f-\mb{E}_{\sigma_n}f\big\|_{L_p(\sigma_n)} \leq C_p \frac{\|\nabla f\|_{L_p(\sigma_n)}}{1+\sqrt{\log\big(\|\nabla f\|_{L_p(\sigma_n)}/\|\nabla f\|_{L_1(\sigma_n)}\big)}}.
\end{equation}
\end{theorem}


\subsection{Negative powers of the Laplacian} Let $(\Omega,\mu)$ be a finite measure space, $(E,\|\cdot\|_E)$ be a Banach space and $p\in[1,\infty]$. If $T:L_p(\mu)\to L_p(\mu)$ is a bounded linear operator, then, by abuse of notation, we will also denote by $T$ its natural $E$-valued extension $T\equiv T\otimes \msf{Id}_E:L_p(\mu;E)\to L_p(\mu;E)$. 

The discrete derivatives~\eqref{eq:partialderivative} on the Hamming cube $\ms{C}_n$ satisfy $\partial_i^2=\partial_i$ for every $i\in\{1,\ldots,n\}$ and thus the hypercube Laplacian is defined as $\Delta\eqdef\sum_{i=1}^n\partial_i$. Note that for $g,h$ on $\ms{C}_n$ with values in $\C$ and $E$ respectively,  we have
\begin{equation}\label{eq:ipp}
\forall \ i\in\{1,\ldots,n\}, \quad \mb{E}_{\sigma_n} [g \, \partial_i h] = \mb{E}_{\sigma_n}[ (\partial_i g) h] =  \mb{E}_{\sigma_n}[ (\partial_i g) (\partial_i h)]
\end{equation}
The formula is also true if $g$ has values in the dual $E^\ast$ and the product is the duality bracket.  The operator $\Delta$ is the (positive) infinitesimal generator of the discrete heat semigroup $\{P_t\}_{t\geq0}$ on $\ms{C}_n$, that is, $P_t=e^{-t\Delta}$ (see, e.g.,~\cite{O'Do14}). Let us mention that functional calculus involving $\Delta$ can be easily expressed using the Walsh basis. This is the case for all Fourier multipliers appearing below which are defined by formula~\eqref{eq:defoperator}. 

All available proofs of Talagrand's inequality~\eqref{eq:talagrand} make crucial use of the hypercontractivity of $\{P_t\}_{t\geq0}$ (first proven by Bonami in~\cite{Bon70}) along with some version of ``orthogonality''~\cite{Tal94} or semigroup identites~\cite{BKS03, CL12} specific to the scalar case. In particular, Talagrand~\cite{Tal94} used Parseval's identity for the Walsh basis to express the variance of a function $f:\ms{C}_n\to\C$ as
\begin{equation} \label{eq:tala-identity}
\mathrm{Var}_{\sigma_n}(f) = \sum_{i=1}^n \big\| \Delta^{-1/2}\partial_i f\big\|_{L_2(\sigma_n)}^2,
\end{equation}
and thus reduced the problem to obtaining effective estimates for $\|\Delta^{-1/2}h\| _{L_2(\sigma_n)}$.
One tool which allows us to circumvent algebraic representations such as \eqref{eq:tala-identity} (see the proof of Theorem~\ref{thm:l1lpscalar} below) are one-sided Riesz transform inequalities, which can combined with certain new vector-valued estimates on negative powers of the generator of the semigroup $\{P_t\}_{t\geq0}$.

Let $\alpha\geq0$. We say that a Banach space $E$ has nontrivial Rademacher type if $E$ has Rademacher type $s$ for some $s\in(1,2]$. It has been proven by Naor and Schechtman~\cite{NS02} that if a Banach space $(E,\|\cdot\|_E)$ has nontrivial Rademacher type, then for every $p\in(1,\infty)$ and $\alpha\in(0,\infty)$, there exists $K_p(\alpha)=K_p(\alpha,E)\in(0,\infty)$ such that for every $n\in\N$ and $f:\ms{C}_n\to E$, we have
\begin{equation} \label{eq:ns02}
\|\Delta^{-\alpha}f\|_{L_p(\sigma_n;E)} \leq K_p(\alpha) \|f\|_{L_p(\sigma_n;E)}.
\end{equation}
Conversely, if~\eqref{eq:ns02} holds true for some $p$ and $\alpha$, then $E$ has nontrivial Rademacher type. The proof of Theorem~\ref{thm:l1lpscalar} relies on the following strengthening of Naor and Schechtman's inequality~\eqref{eq:ns02}.

\begin{theorem} \label{thm:naorschechtman}
Let $(E,\|\cdot\|_E)$ be a Banach space of nontrivial Rademacher type. Then, for every $p\in(1,\infty)$ and $\alpha\in(0,\infty)$, there exists $K_p(\alpha)=K_p(\alpha,E)\in(0,\infty)$ such that for every $n\in\N$ and $f:\ms{C}_n\to E$, we have
\begin{equation} \label{eq:thmnaorschechtman}
\|\Delta^{-\alpha}f\|_{L_p(\sigma_n;E)} \leq K_p(\alpha) \frac{\|f\|_{L_p(\sigma_n;E)}}{1+\log^\alpha\big(\|f\|_{L_p(\sigma_n;E)}/\|f\|_{L_1(\sigma_n;E)}\big)}.
\end{equation}
\end{theorem}

We note in passing that when $E=\C$, $\alpha=\tfrac{1}{2}$ and $p=2$, Theorem~\ref{thm:naorschechtman} had been proven in~\cite[Proposition~2.3]{Tal94}. However, Talagrand's argument heavily uses orthogonality via Parseval's identity for the Walsh basis and is unlikely to work in the vector-valued setting which is of interest here.


\subsection{Vector-valued multipliers and inequalities involving Orlicz norms} In Talagrand's original work~\cite{Tal94}, he observed that~\eqref{eq:talagrand} admits a strengthening in terms of Orlicz norms (see~\cite{RR91}). Recall that if $\psi:[0,\infty)\to[0,\infty)$ is a Young function, i.e. a convex function satisfying
\begin{equation}
\lim_{x\to0}\frac{\psi(x)}{x}=0 \quad \mbox{and} \quad \lim_{x\to\infty}\frac{\psi(x)}{x}=\infty,
\end{equation}
and $(E,\|\cdot\|_E)$ is a Banach space, then the $\psi$-Orlicz norm of a function $f:\ms{C}_n\to E$ is given by
\begin{equation}
\|f\|_{L_\psi(\sigma_n;E)} \eqdef \inf\Big\{t\geq0: \ \int_{\ms{C}_n} \psi\big(\|f(\e)\|_E / t\big)\diff\sigma_n(\e) \leq 1\Big\}.
\end{equation}
It is evident that for $\psi(t)=t^p$, we have $\|\cdot\|_{L_\psi(\sigma_n;E)}=\|\cdot\|_{L_p(\sigma_n;E)}$. More generally, for $p\in(1,\infty)$ and $r\in\R$ we will denote by $\|\cdot\|_{L_p(\log L)^r(\sigma_n;E)}$ the Orlicz norm correspoding to a Young function $\psi_{p,r}$ with $\psi_{p,r}(x) = x^p \log^r(e+x)$ for $x$ large enough (to ensure convexity of $\psi_{p,r}$ when $r<0$).

In~\cite[Theorem~1.6]{Tal94}, Talagrand showed that~\eqref{eq:talagrand} can be strengthened as follows. There exists a universal constant $C\in(0,\infty)$ such that for every $n\in\N$, every function $f:\ms{C}_n\to\C$ satisfies
\begin{equation} \label{eq:talagrandorlicz}
\big\| f-\mb{E}_{\sigma_n}f\big\|_{L_2(\sigma_n)}^2 \leq C \sum_{i=1}^n \|\partial_if\|_{L_2(\log L)^{-1}(\sigma_n)}^2.
\end{equation}
It is in fact true (see~\cite[Lemma~2.5]{Tal94} or Lemma~\ref{lem:orlicz-upper} below) that~\eqref{eq:talagrandorlicz} formally implies~\eqref{eq:talagrand}. In this direction we can prove the following strengthening of Theorem~\ref{thm:useivv}.

\begin{theorem} \label{thm:useivvorlicz}
Let $(E,\|\cdot\|_E)$ be a Banach space with Rademacher type 2. Then, there exists $C=C(E)\in(0,\infty)$ such that for every $\e\in(0,1)$ and $n\in\N$, every function $f:\ms{C}_n\to E$ satisfies
\begin{equation} \label{eq:thmuseivvorlicz}
\big\| f-\mb{E}_{\sigma_n}f\big\|_{L_2(\sigma_n;E)}^2 \leq \frac{C}{\e} \sum_{i=1}^n \|\partial_if\|_{L_2(\log L)^{-1+\e}(\sigma_n;E)}^2.
\end{equation}
\end{theorem}

Furthermore, the proofs of Theorem~\ref{thm:useeldan} in fact yield the following improvement of~\eqref{eq:thmuseeldan}, which extends~\eqref{eq:talagrandorlicz} to spaces of martingale type 2.

\begin{theorem} \label{thm:useeldanorlicz}
Let $(E,\|\cdot\|_E)$ be a Banach space with martingale type 2. Then, there exists $C=C(E)\in(0,\infty)$ such that for every $n\in\N$, every function $f:\ms{C}_n\to E$ satisfies
\begin{equation} \label{eq:thmuseeldanorlicz}
\big\| f-\mb{E}_{\sigma_n}f\big\|_{L_2(\sigma_n;E)}^2 \leq C \sum_{i=1}^n \|\partial_if\|_{L_2(\log L)^{-1}(\sigma_n;E)}^2.
\end{equation}
\end{theorem}

We now turn to Orlicz space strengthenings of Theorem~\ref{thm:naorschechtman}. The scalar-valued analogue of this problem had first been studied by Feissner~\cite{Fei75} and was later completely settled by Bakry and Meyer~\cite{BM82}, who showed the following. For every $p\in(1,\infty)$ and $\alpha\in(0,\infty)$ there exists $K_p(\alpha)\in(0,\infty)$ such that for every $n\in\N$ and $f:\ms{C}_n\to\C$,
\begin{equation} \label{eq:bakrymeyer}
\|\Delta^{-\alpha}f\|_{L_p(\sigma_n)} \leq K_p(\alpha) \|f\|_{L_p(\log L)^{-p\alpha}(\sigma_n)}.
\end{equation}
In~\cite{BM82}, inequality~\eqref{eq:bakrymeyer} is stated and proven for the generator of the Ornstein--Uhlenbeck semigroup on Gauss space, yet straightforward modifications of the proof show that~\eqref{eq:bakrymeyer} holds for the generator of a general hypercontractive semigroup. While proving~\eqref{eq:bakrymeyer} with the Orlicz norm on the right hand side replaced by $L_p(\log L)^{-r}(\sigma_n)$ for $r<p\alpha$ is fairly simple (see~\cite[Th\'eor\`eme 5]{BM82}), obtaining the result with the optimal Orlicz space $L_p(\log L)^{-p\alpha}(\sigma_n)$ is more delicate. In~\cite[Th\'eor\`eme 6]{BM82} this is achieved via a complex interpolation scheme relying on Littlewood--Paley--Stein theory~\cite{Ste70} (in the form of bounds for the imaginary Riesz potentials $\Delta^{it}$, where $t\in\R$). Even though such tools are generally not available for functions with values in a general Banach space of nontrivial type (see, e.g.,~\cite{G-D91, Xu98, Hyt07}), we prove the following theorem.

\begin{theorem} [Vector-valued Bakry--Meyer inequality] \label{thm:vectorbakrymeyer}
Let $(E,\|\cdot\|_E)$ be a Banach space of nontrivial Rademacher type. Then, for every $p\in(1,\infty)$ and $\alpha\in(0,\infty)$, there exists $K_p(\alpha)=K_p(\alpha,E)\in(0,\infty)$ such that for every $n\in\N$ and $f:\ms{C}_n\to E$, we have
\begin{equation} \label{eq:vectorbakrymeyer}
\|\Delta^{-\alpha}f\|_{L_p(\sigma_n;E)} \leq K_p(\alpha) \|f\|_{L_p(\log L)^{-p\alpha}(\sigma_n;E)}.
\end{equation}
\end{theorem}

It will be shown in Lemma~\ref{lem:orlicz-upper} below, that Theorem~\ref{thm:vectorbakrymeyer} is indeed a strengthening of Theorem~\ref{thm:naorschechtman}. In view of the result of~\cite{NS02}, it is evident that the assumption that the target space $E$ has nontrivial type is both necessary and sufficient in Theorem~\ref{thm:vectorbakrymeyer}. While the ingredients used in the proof of~\cite[Th\'eor\`eme 6]{BM82} cannot be applied in the vector-valued setting of Theorem~\ref{thm:vectorbakrymeyer},~\eqref{eq:vectorbakrymeyer} will be proven as a consequence of the scalar inequality~\eqref{eq:bakrymeyer} using the \mbox{following vector-valued multiplier theorem.}

\begin{theorem} \label{thm:multiplier}
Let $(E,\|\cdot\|_E)$ be a Banach space of nontrivial Rademacher type and consider a holomorphic function $h:\mb{D}_r\to\C$ where $\mb{D}_r = \{z\in\C: \ |z|<r\}$, where $r\in(0,\infty)$. Then, for every $\alpha\in(0,\infty)$ and $p\in(1,\infty)$, there exists a constant $C_h(\alpha,p)=C_h(\alpha,p,E)\in(0,\infty)$ such that for every $n\in\N$, every function $f:\ms{C}_n\to E$ satisfies
\begin{equation} \label{eq:thmmultiplier}
\big\| h\big( \Delta^{-\alpha}\big) f\big\|_{L_p(\sigma_n;E)} \leq C_h(\alpha,p) \|f\|_{L_p(\sigma_n;E)}.
\end{equation}
\end{theorem}

When $E=\C$, Theorem~\ref{thm:multiplier} is a classical result of Meyer~\cite[Th\`eor\'eme 3]{Mey84}. The vector-valued extension presented here crucially relies on the bounds on the action of negative powers of $\Delta$ on vector-valued tail spaces obtained by Mendel and Naor in~\cite{MN14}.

\subsection{Talagrand metric spaces} 

The vector-valued discrete Poincar\'e inequality~\eqref{eq:vectorpoincare} is intimately connected to a metric version of Rademacher type, called Enflo type (see~\cite{Enf78, NS02}). In view of this connection, we introduce the following metric invariant, inspired by Talagrand's inequality~\eqref{eq:talagrandorlicz}.

\begin{definition} [Talagrand type]  \label{def:talagrand}
Let $\psi:[0,\infty)\to[0,\infty)$ be a Young function and $p\in(0,\infty)$. We say that a metric space $(\MM,d_\MM)$ has Talagrand type $(p,\psi)$ with constant $\tau\in(0,\infty)$ if for every $n\in\N$, every function $f:\ms{C}_n\to\MM$ satisfies
\begin{equation} \label{eq:deftalagrand}
\int_{\ms{C}_n\times\ms{C}_n} d_\MM\big(f(\e),f(\delta)\big)^p\diff\sigma_{2n}(\e,\delta) \leq \tau^p \sum_{i=1}^n \|\mathfrak{d}_if\|^p_{L_\psi(\sigma_n)},
\end{equation}
where $\mathfrak{d}_if:\ms{C}_n\to\R_+$ is given by
\begin{equation} \label{eq:metricderivative}
\forall \ \e\in\ms{C}_n, \qquad\mathfrak{d}_if(\e) = \tfrac{1}{2} d_\MM\big(f(\e),f(\e_1,\ldots,\e_{i-1},-\e_i,\e_{i+1},\ldots,\e_n)\big).
\end{equation}
\end{definition}

It is clear that if $(E,\|\cdot\|_E)$ is a Banach space then $\|\partial_if(\e)\|_E$ coincides with $\mathfrak{d}_if(\e)$. It can be easily seen that if a Banach space $E$ has\mbox{ the property that for every $n\in\N$, every $f:\ms{C}_n\to E$ satisfies}
\begin{equation} \label{eq:orlicztalagrandtype}
\big\|f-\mb{E}_{\sigma_n}f\big\|_{L_p(\sigma_n;E)}^p \leq \tau_\ast^p \sum_{i=1}^n \|\partial_if\|^p_{L_\psi(\sigma_n;E)}
\end{equation}
for some $\tau_\ast\in(0,\infty)$, then $E$ also has Talagrand type $(p,\psi)$. Indeed, applying \eqref{eq:orlicztalagrandtype} to the function $F:\ms{C}_n\times\ms{C}_n\to E$ given by $F(\e,\delta) = f(\e)-f(\delta)$ which has $\mb{E}_{\sigma_{2n}}F =0$, we get
\begin{equation*}
\begin{split}
\int_{\ms{C}_n\times\ms{C}_n} \! \|f(\e)-f(\delta)\|_E^p&\diff\sigma_{2n}(\e,\delta)  = \big\|F-\mb{E}_{\sigma_{2n}}F\big\|_{L_p(\sigma_{2n};E)}^p \\ & \leq \tau_\ast^p \sum_{i=1}^n \big(  \|\partial_{\e_i}F\|^p_{L_\psi(\sigma_{2n};E)} +  \|\partial_{\delta_i}F\|^p_{L_\psi(\sigma_{2n};E)} \big) = 2\tau_*^p \sum_{i=1}^n \|\partial_if\|^p_{L_\psi(\sigma_n;E)},
\end{split}
\end{equation*}
and thus $E$ has Talagrand type $(p,\psi)$ with constant $\tau\leq 2^{1/p}\tau_\ast$. Hence, Theorems~\ref{thm:useivvorlicz} and~\ref{thm:useeldanorlicz} can both be translated as implications of Talagrand type from Rademacher and martingale type respectively (see also the discussion in Section~\ref{sec:8}). It is worth investigating whether natural examples of nonlinear metric spaces (e.g.~Alexandrov spaces of nonpositive or nonnegative curvature, transportation cost spaces and others) have Talagrand type. In this direction, we prove the following Talagrand-type inequality for functions with values in Gromov hyperbolic groups. For $p\in[1,\infty)$ and $\delta\in[0,1]$, let $\psi_{p,\delta}:[0,\infty)\to\R$ be a Young function with $\psi_{p,\delta}(x) = t^p \log^{-\delta}(e+x)$ for $x$ large enough. 

\begin{theorem} \label{thm:gromov}
There exists $\tau\in(0,\infty)$ such that every $\e\in(0,1)$ the following holds. Every Gromov hyperbolic group $\msf{G}$ equipped with the shortest path metric on the Cayley graph with respect to a finite generating set $\msf{S}\subseteq\msf{G}$ has Talagrand type $(2,\psi_{2,1-\e})$ with constant $\tau/\sqrt{\e}$.
\end{theorem}

The proof of Theorem~\ref{thm:gromov} relies on a result of Ostrovskii~\cite{Ost14}, according to which the Cayley graph of every Gromov hyperbolic groups admits a bi-Lipschitz embeddng in an arbitrary nonsuperreflexive Banach space, combined with a classical construction of James~\cite{Jam78}.

We say that a Riemannian manifold has pinched negative curvature if its sectional curvature takes values in the interval $[-R,-r]$ for some $r,R\in(0,\infty)$ with $r<R$. After the proof of Theorem~\ref{thm:gromov} in Section~\ref{sec:gromov}, we also prove the following result.

\begin{theorem} \label{thm:pinched}
Let $n\in\N$ and $(\msf{M},g)$ be an $n$-dimensional complete, simply connected Riemannian manifold with pinched negative curvature. Then, for every $\e\in(0,1)$, $(\msf{M},d_\msf{M})$ has Talagrand type $(2,\psi_{2,1-\e})$ with constant $\tau/\sqrt{\e}$ where $\tau$ depends only on $n$ and the parameters $r,R$.
\end{theorem}

Theorems~\ref{thm:gromov} and~\ref{thm:pinched} describe two classes of nonpositively curved spaces which satisfy a Talagrand-type inequality that strengthens Enflo type 2. It remains an intriguing open problem to understand whether every CAT(0) space has this property (see also Section~\ref{sec:8}).

\subsection{Embeddings of nonlinear quotients of the cube and Talagrand type}
Let $(\MM,d_\MM)$ and $(\NN,d_\NN)$ be metric spaces. A function $f:\MM\to\NN$ has bi-Lipschitz distortion at most $D\geq1$ if there exists $s\in(0,\infty)$ such that
\begin{equation}
\forall \ x,y\in\MM, \quad sd_\MM(x,y) \leq d_\NN\big(f(x),f(y)\big) \leq sD d_\MM(x,y).
\end{equation}
We will denote by $c_\NN(\MM)$ the infimal bi-Lipschitz distortion of a function $f:\MM\to\NN$. When $\NN = L_p(\R)$, we will abbreviate $c_{L_p(\R)}(\MM)$ as $c_p(\MM)$. Consider the hypercube $\ms{C}_n$ endowed with the Hamming metric $\rho(\e,\delta) = \|\e-\delta\|_1$. The geometric significance of Enflo type stems (partially) from the fact (see \cite{NS02}) that if a metric space $\MM$ has Enflo type $p$ with constant  $T\in(0,\infty)$, then
\begin{equation} \label{type-no-embed}
c_\MM(\ms{C}_n) \geq T^{-1} n^{1-\frac{1}{p}}.
\end{equation}
In this section, we will establish a more delicate bi-Lipschitz nonembeddability property which is a consequence of the Talagrand type inequality \eqref{eq:deftalagrand}.

Let $\ms{R}\subseteq \ms{C}_n\times\ms{C}_n$ be an arbitrary equivalence relation and denote by $\ms{C}_n/\ms{R}$ the set of all equivalence classes of $\ms{R}$ equipped with the quotient metric, which is given by
\begin{equation}
\forall \ [\e],[\delta]\in\ms{C}_n/\ms{R}, \quad \rho_{\ms{C}_n/\ms{R}} \big( [\e], [\delta] \big) \eqdef \min\big\{ \rho(\eta_1,\zeta_1)+\cdots+\rho(\eta_{k},\zeta_k)\big\};
\end{equation}
here the minimum is taken over all $k\geq1$ and $\eta_1,\ldots,\eta_k,\zeta_1,\ldots,\zeta_k\in\ms{C}_n$ with $\eta_1\in[\e]$, $\zeta_k\in[\delta]$ and $[\zeta_j] = [\eta_{j+1}]$ for every $j\in\{1,\ldots,k-1\}$. We shall now present an implication of Talagrand type on embeddings of nonlinear quotients\footnote{The term ``nonlinear'' here is meant to emphasize the distinction between quotients of the hypercube with respect to an arbitrary equivalence relation and quotients by linear codes (see \cite{MS77} and Remark \ref{rem:khot-naor} below). Recall that if we identify $\ms{C}_n$ with $\mb{F}_2^n$, where $\mb{F}_2$ is the field with two elements, a linear code is an $\mb{F}_2$-subspace $C\subseteq\ms{C}_n$ and the corresponding quotient is the $\mb{F}_2$-vector space $\mb{F}_2^n/C$ endowed with the quotient metric.} of the cube which strengthens the corresponding bounds that one can deduce from Enflo type. We will denote by $\partial_i \ms{R}$ the boundary of $\ms{R}$ in \mbox{the direction $i$, that is}
\begin{equation} \label{eq:boundaries}
\forall \ i\in\{1,\ldots,n\}, \quad \partial_i \ms{R} \eqdef \big\{ \e\in\ms{C}_n: \ \big(\e, (\e_1,\ldots,\e_{i-1},-\e_i,\e_{i+1},\ldots,\e_n)\big) \notin \ms{R}\big\}
\end{equation}
and by $\msf{a}_p(\ms{R})$ the quantity
\begin{equation}
\msf{a}_p(\ms{R}) \eqdef \left( \int_{\ms{C}_n\times\ms{C}_n} \rho_{\ms{C}_n/\ms{R}} \big( [\e], [\delta]\big)^p \diff\sigma_{2n}(\e,\delta)\right)^{1/p}.
\end{equation}

\begin{theorem} \label{thm:embed}
Fix $p\in(0,\infty)$ and a Young function $\psi:[0,\infty)\to[0,\infty)$. If a metric space $(\MM,d_\MM)$ has Talagrand type $(p,\psi)$ with constant $\tau\in(0,\infty)$ then, for every $n\in\N$ and every equivalence relation $\ms{R}\subseteq\ms{C}_n\times\ms{C}_n$, we have
\begin{equation} \label{eq:thm-embed}
c_\MM\big(\ms{C}_n/\ms{R}\big) \geq \frac{2 \tau^{-1}\msf{a}_p(\ms{R})}{\big(\sum_{i=1}^n \psi^{-1} \big(\sigma_n(\partial_i\ms{R})^{-1} \big)^{-p} \big)^{1/p}}.
\end{equation}
\end{theorem}

It is worth noting that in the setting of Theorem \ref{thm:embed}, if $\MM$ has Talagrand type $(p,t\mapsto t^p)$ with constant $\tau$ (a property which is very closely related to Enflo type $p$, see Remark \ref{rem:enflo-def}), then
\begin{equation} \label{eq:weak-embed}
c_\MM\big(\ms{C}_n/\ms{R}\big) \geq \frac{2 \tau^{-1}\msf{a}_p(\ms{R})}{\left(\sum_{i=1}^n \sigma_n(\partial_i\ms{R})  \right)^{1/p}}.
\end{equation}
This estimate, which generalizes \eqref{type-no-embed}, is substantially weaker than \eqref{eq:thm-embed} when $\psi(t)<\!\!\!<t^p$ for large values of $t$. In particular, this is the case for Banach spaces of Rademacher or martingale type $p$ (see Theorems \ref{thm:gen-tal1} and \ref{thm:gen-tal2}). It is also worth mentioning that, in view of Theorem \ref{thm:L1} below, Theorem \ref{thm:embed} provides nontrivial distortion lower bounds even for bi-Lipschitz embeddings into $L_1(\mu)$ spaces.

Theorem \ref{thm:embed} is reminiscent of the celebrated theorem of Kahn, Kalai and Linial \cite{KKL88}, which asserts that there exists a constant $c\in(0,\infty)$ such that for every Boolean function $f:\ms{C}_n\to\{0,1\}$,
\begin{equation} \label{eq:kkl}
\max_{i\in\{1,\ldots,n\}} \|\partial_i f\|_{L_2(\sigma_n)}^2 \geq \frac{c\log n}{n} \mathrm{Var}_{\sigma_n}f = \frac{c\log n}{n} p(1-p),
\end{equation}
where $p=\mb{E}_{\sigma_n}f$. Viewing $f$ as a voting scheme, \eqref{eq:kkl} asserts that if all influences $\|\partial_i f\|_{L_2(\sigma_n)}^2$ are small, then $f$ is necessarily an unfair system in the sense that its expectation is very close to either 0 or 1. Inequality \eqref{eq:thm-embed} puts forth a similar phenomenon in embedding theory:~if all {\it geometric influences} $\sigma_n(\partial_i\ms{R})$ of the partition are small, then the quotient $\ms{C}_n/\ms{R}$ is incompatible with the geometry of the target space $\MM$. Moreover, the quantitative improvement \eqref{eq:thm-embed} of  \eqref{eq:weak-embed} is in direct analogy with the improvement that the KKL inequality \eqref{eq:kkl} offers to the weaker estimate 
$$\max_{i\in\{1,\ldots,n\}} \|\partial_i f\|_{L_2(\sigma_n)}^2 \geq \tfrac{1}{n}\mathrm{Var}_{\sigma_n}f,$$
which follows readily from the Poincar\'e inequality \eqref{eq:poincare} for any function $f:\ms{C}_n\to\C$.

\subsection*{Organization of the paper} In Section~\ref{sec:2}, we will present some elementary inequalities and properties of Orlicz norms which we shall use in the sequel. Section~\ref{sec:3} contains the proof of Theorems~\ref{thm:useivv} and~\ref{thm:useivvorlicz} and Section~\ref{sec:4} contains two proofs of Theorems~\ref{thm:useeldan} and~\ref{thm:useeldanorlicz}, one using stochastic calculus and one Fourier analytic. In Section~\ref{sec:5}, we prove Theorems~\ref{thm:lptalagrand} and~\ref{thm:l1lpscalar} and their analogue in Gauss space, Theorem~\ref{thm:usemp}, by a combination of semigroup methods and Riesz transforms. Section~\ref{sec:6} contains the proof of Theorem~\ref{thm:multiplier} and the derivation of Theorems~\ref{thm:naorschechtman} and~\ref{thm:vectorbakrymeyer}. In Section~\ref{sec:gromov} we present the proof of Theorems~\ref{thm:gromov} and~\ref{thm:pinched} and in Section~\ref{sec:embed} we present the proof of the nonembeddability result of Theorem \ref{thm:embed}. Finally, Section~\ref{sec:8} contains some \mbox{concluding remarks and open problems.}

\bigskip

\noindent{\bf Acknowledgements.} We are grateful to Florent Baudier, Michel Ledoux, Assaf Naor, Sang Woo Ryoo and Ramon van Handel for helpful discussions and feedback. We would also like to thank the anonymous referee for communicating the content of Remark \ref{rem:referee} to us.


\section{Some preliminary calculus lemmas} \label{sec:2}

In this section, we present a few elementary facts  related to Orlicz norms which we shall repeatedly use in the sequel. While these results are central for our proofs, they are mostly technical and therefore can be skipped on first reading. We gather them here in order to avoid digressions in the main part of the text.

\begin{lemma} \label{lem:orlicz-withpower}
Let $(E,\|\cdot\|_E)$ be a Banach space and $(\Omega,\mu)$ a probability space. For every $r\in(1,\infty)$, $\gamma,\eta\in(0,\infty)$ and $\e\in[0,1)$, there exists $A=A(r,\gamma,\eta,\e)\in(0,\infty)$ such that every $h:\Omega\to E$ satisfies,
\begin{equation}
\int_0^\infty e^{-\eta t} \big\|h\big\|_{L_{1+(r-1)e^{-\gamma t}}(\mu;E)}^r \frac{\diff t}{t^\e} \leq A \|h\|^r_{L_r(\log L)^{-1+\e}(\mu;E)}.
\end{equation}
\end{lemma}

\begin{proof}
Since both sides only depend on the norm of $h$, we can assume that $E=\C$ and $h\geq0$. Moreover, without loss of generality $\eta\leq1=\gamma$. Suppose, by homogeneity, that the right hand side satisfies $\|h\|_{L_r(\log L)^{-1+\e}(\mu)}\leq1$, which implies that
\begin{equation*}
\int_{\Omega} \frac{h^r}{\log^{1-\e}(e+h)} \diff\mu \leq 1.
\end{equation*}
For $k\geq1$, let $h_k = h \cdot {\bf 1}_{\{2^{k-1}<h\leq 2^k\}}$ and $h_0 = h\cdot {\bf 1}_{\{h\leq1\}}$, so that
\begin{equation} \label{eq:boundedorlicz1}
\sum_{k=0}^\infty \frac{1}{(k+1)^{1-\e}} \int_\Omega h_k^r \diff\mu \leq 1.
\end{equation}
Moreover, observe that
\begin{equation*}
\int_0^\infty e^{-\eta t} \big\|h\big\|_{L_{1+(r-1)e^{-t}}(\mu)}^r \frac{\diff t}{t^\e} \leq \int_0^\infty e^{-\eta t} \big\|h\big\|_{L_{1+(r-1)e^{-\eta t}}(\mu)}^r \frac{\diff t}{t^\e} \asymp_{r,\eta} \int_1^r \|h\|_{L_\nu(\mu)}^r \frac{\diff \nu}{(r-\nu)^\e},
\end{equation*}
where the inequality follows from the monotonicity of $L_s(\mu)$ norms and the equivalence by the change of variables $\nu=1+(r-1)e^{-\eta t}$. 

The right hand side then satisfies
\begin{equation*}
\begin{split}
\int_1^r & \|h\|_{L_\nu(\mu)}^r \frac{\diff \nu}{(r-\nu)^\e} = \int_1^r \Big( \sum_{k=0}^\infty \int_\Omega h_k^\nu \diff\mu\Big)^{r/\nu} \frac{\diff\nu}{(r-\nu)^\e} 
\\ & \leq 2^r \int_1^r \Big(\sum_{k=0}^\infty 2^{-(r-\nu)k} \int_\Omega h_k^r\diff\mu\Big)^{r/\nu} \frac{\diff\nu}{(r-\nu)^\e}
\\ & \stackrel{\eqref{eq:boundedorlicz1}}{\leq} 2^r \int_1^r \sum_{k=0}^\infty (k+1)^{\frac{r(1-\e)}{\nu}} 2^{-\frac{(r-\nu)kr}{\nu}} \frac{1}{(k+1)^{1-\e}} \int_\Omega h_k^r\diff\mu \frac{\diff\nu}{(r-\nu)^\e}
\\ & = 2^r \sum_{k=0}^\infty \left( \int_1^r (k+1)^{\frac{r(1-\e)}{\nu}} 2^{-\frac{(r-\nu)kr}{\nu}} \frac{\diff\nu}{(r-\nu)^\e} \right)\frac{1}{(k+1)^{1-\e}} \int_\Omega h_k^r\diff\mu
\\ & \stackrel{\eqref{eq:boundedorlicz1}}{\leq} 2^r \max_{k\geq0}\left\{ \int_1^r (k+1)^{\frac{r(1-\e)}{\nu}} 2^{-\frac{(r-\nu)kr}{\nu}} \frac{\diff\nu}{(r-\nu)^\e}\right\}
\end{split},
\end{equation*}
where the second inequality follows from Jensen's inequality for the convex function $t\mapsto t^{r/\nu}$ with weights~\eqref{eq:boundedorlicz1}. Now, by multiplying $k$ by $r$, one can easily see that
\begin{equation*}
\begin{split}
\max_{k\geq0}\left\{ \int_1^r (k+1)^{\frac{r(1-\e)}{\nu}} 2^{-\frac{(r-\nu)kr}{\nu}} \frac{\diff\nu}{(r-\nu)^\e}\right\} &  \asymp_{r,\e} \max_{k\geq0}\left\{ \int_1^r k^{\frac{r(1-\e)}{\nu}} e^{-(r-\nu)k} \frac{\diff\nu}{(r-\nu)^\e}\right\}
\\ & \asymp_{r,\e} \max_{k\geq0}\left\{ \int_{1/r}^1 k^{\frac{1-\e}{u}} e^{-(1-u)k} \frac{\diff u}{(1-u)^\e}\right\},
\end{split}
\end{equation*}
where the second equivalence follows by the change of variables $u=\nu/r$ and a further change of variables in $k$. For $k\geq0$ and $\e\in(0,1)$, write
\begin{equation*}
 \int_{1/r}^1 k^{\frac{1-\e}{u}} e^{-(1-u)k} \frac{\diff u}{(1-u)^\e} = \underbrace{\int_{1/r}^{1-1/k} k^{\frac{1-\e}{u}} e^{-(1-u)k} \frac{\diff u}{(1-u)^\e}}_{I_k(\e)} + \underbrace{\int_{1-1/k}^1 k^{\frac{1-\e}{u}} e^{-(1-u)k} \frac{\diff u}{(1-u)^\e}}_{J_k(\e)}
\end{equation*}
and notice that
\begin{equation*}
J_k(\e) \leq k^{\frac{1-\e}{1-1/k}} \int_{1-1/k}^1 \frac{\diff u}{(1-u)^\e} = \frac{1}{1-\e} k^{\frac{(1-\e)k}{k-1}-(1-\e)} \asymp_\e 1.
\end{equation*}
Moreover, if $u\leq 1-\tfrac{1}{k}$, then
\begin{equation*}
k^{-\frac{\e}{u}} \frac{1}{(1-u)^\e} \leq k^{\e(1-\frac{1}{u})} <1,
\end{equation*}
which implies that
\begin{equation*}
I_k(\e) \leq \int_{1/r}^{1-1/k} k^{\frac{1}{u}} e^{-(1-u)k} \diff u  \leq \int_{1/r}^1 k^{\frac{1}{u}} e^{-(1-u)k} \diff u \eqdef R_k.
\end{equation*}
Finally, to bound $R_k$, we integrate by parts
\begin{equation*}
\begin{split}
R_k= \int_{1/r}^1 k^{\frac{1}{u}} \big(\tfrac{e^{-(1-u)k}}{k}\big)'\diff u  & = 1 - k^{r-1} e^{-(1-1/r)k} + \frac{\log k}{k} \int_{1/r}^1 k^{\frac{1}{u}} e^{-(1-u)k} \frac{\diff u}{u^2}
\\ & \leq 1- k^{r-1} e^{-(1-1/r)k} + \frac{r^2\log k}{k} R_k,
\end{split}
\end{equation*}
which, after rearranging, readily implies that $R_k\lesssim_r 1$ and the proof is complete.
\end{proof}

Using Hölder's inequality, we can easily deduce the following variant of Lemma~\ref{lem:orlicz-withpower} which we will need to prove Theorems~\ref{thm:lptalagrandorlicz} and~\ref{thm:lpmporlicz} below.

\begin{lemma} \label{lem:orlicz-nopower}
Let $(E,\|\cdot\|_E)$ be a Banach space and $(\Omega,\mu)$ a probability space. For every $r\in(1,\infty)$, $\gamma, \eta\in(0,\infty)$ and $\e\in[0,1)$, there exists $B=B(r,\gamma,\eta,\e)\in(0,\infty)$ such that for every $\theta\in(0,1)$, every $h:\Omega\to E$ satisfies
\begin{equation}
\int_0^\infty e^{-\eta t} \big\|h\big\|_{L_{1+(r-1)e^{-\gamma t}}(\mu;E)} \frac{\diff t}{t^\e} \leq \frac{B}{\theta^\frac{r-1}{r}} \|h\|_{L_r(\log L)^{-r(1-\e)+\theta}(\mu;E)}.
\end{equation}
\end{lemma}

\begin{proof}
Without loss of generality, we will again assume that $E=\C$, $h\geq0$ and $\eta\leq1=\gamma$. As in the proof of Lemma~\ref{lem:orlicz-withpower}, a change of variables shows that
\begin{equation} \label{eq:orlicz-nopower1}
\int_0^\infty e^{-\eta t} \big\|h\big\|_{L_{1+(r-1)e^{-t}}(\mu)} \frac{\diff t}{t^\e} \leq \int_0^\infty e^{-\eta t} \big\|h\big\|_{L_{1+(r-1)e^{-\eta t}}(\mu)} \frac{\diff t}{t^\e} \asymp_{r,\eta} \int_1^r \|h\|_{L_\nu(\mu)} \frac{\diff \nu}{(r-\nu)^\e}.
\end{equation}
Fix $\theta\in(0,1)$. By Hölder's inequality, we have
\begin{equation*}
 \int_1^r \|h\|_{L_\nu(\mu)} \frac{\diff \nu}{(r-\nu)^\e} \leq \Big( \int_1^r \|h\|_{L_\nu(\mu)}^r \frac{\diff\nu}{(r-\nu)^{1-r(1-\e)+\theta}} \Big)^{1/r} \Big( \int_1^r \frac{\diff \nu}{(r-\nu)^{1-\frac{\theta}{r-1}}} \Big)^{(r-1)/r}
\end{equation*}
and since $\int_1^r \frac{\diff \nu}{(r-\nu)^{1-\frac{\theta}{r-1}}} \asymp_r  \tfrac{1}{\theta}$, we deduce from Lemma~\ref{lem:orlicz-withpower} that
\begin{equation*}
 \int_1^r \|h\|_{L_\nu(\mu)} \frac{\diff \nu}{(r-\nu)^\e} \leq \frac{A}{\theta^\frac{r-1}{r}} \big\|h\big\|_{L_r(\log L)^{-r(1-\e)+\theta}(\mu)}
\end{equation*}
for some $A=A(r,\e)$. Then, the proof is complete by~\eqref{eq:orlicz-nopower1}.
\end{proof}

The following lemma will be used to prove Theorems~\ref{thm:lptalagrand} and~\ref{thm:naorschechtman}.

\begin{lemma} \label{lem:log-nopower}
Let $(E,\|\cdot\|_E)$ be a Banach space and $(\Omega,\mu)$ a probability space. For every $r\in[1,\infty)$, $\gamma, \eta\in(0,\infty)$ and $\e\in[0,1)$, there exists $C=C(r,\gamma, \eta,\e)\in(0,\infty)$ such that every $h:\Omega\to E$ satisfies,
\begin{equation} \label{eq:log-nopower0}
\int_0^\infty e^{-\eta t} \big\|h\big\|_{L_{1+(r-1)e^{-\gamma t}}(\mu;E)} \frac{\diff t}{t^\e} \leq C\ \frac{\|h\|_{L_r(\mu;E)}}{1+\log^{1-\e}\big(\|h\|_{L_r(\mu;E)}/\|h\|_{L_1(\mu;E)}\big)}.
\end{equation}
\end{lemma}

\begin{proof}
Without loss of generality, we will again assume that $E=\C$, $h\geq0$ and $\eta\leq1=\gamma$. As in the proof of Lemma~\ref{lem:orlicz-withpower}, a change of variables shows that
\begin{equation} \label{eq:log-nopower1}
\int_0^\infty e^{-\eta t} \big\|h\big\|_{L_{1+(r-1)e^{-t}}(\mu)} \frac{\diff t}{t^\e} \leq \int_0^\infty e^{-\eta t} \big\|h\big\|_{L_{1+(r-1)e^{-\eta t}}(\mu)} \frac{\diff t}{t^\e} \asymp_{r,\eta} \int_1^r \|h\|_{L_\nu(\mu)} \frac{\diff \nu}{(r-\nu)^\e}.
\end{equation}
By Hölder's inequality, if $\theta(\nu) = \tfrac{r-\nu}{\nu(r-1)}$ is such that $\tfrac{1-\theta}{r}+\tfrac{\theta}{1}=\tfrac{1}{\nu}$, then
\begin{equation} \label{eq:log-nopower2}
\int_1^r \|h\|_{L_\nu(\mu)} \frac{\diff \nu}{(r-\nu)^\e} \leq \|h\|_{L_r(\mu)} \int_1^r b^{\theta(\nu)} \frac{\diff\nu}{(r-\nu)^\e} \asymp_{r,\e} \|h\|_{L_r(\mu)} \int_0^1 b^\theta \frac{\diff \theta}{\theta^\e},
\end{equation}
where $b=\|h\|_{L_1(\mu)}/\|h\|_{L_r(\mu)} \in(0,1]$. Finally, if $b<1$, notice that
\begin{equation*}
 \int_0^1 b^\theta \frac{\diff \theta}{\theta^\e} = \int_0^1 e^{-\theta\log(1/b)} \frac{\diff \theta}{\theta^\e} = \frac{1}{\log^{1-\e}(1/b)} \int_0^{\log(1/b)} e^{-u}\frac{\diff u}{u^\e} \lesssim_\e \frac{1}{\log^{1-\e}(1/b)}
\end{equation*}
and the conclusion follows from~\eqref{eq:log-nopower1} and~\eqref{eq:log-nopower2}.
\end{proof}

The following lemma shows that the Orlicz norm statements of Theorems~\ref{thm:useivvorlicz} and~\ref{thm:useeldanorlicz} indeed strengthen Theorems~\ref{thm:useivv} and~\ref{thm:useeldan} respectively. In the special case $r=2$ and $s=1$, this has been proven by Talagrand in~\cite[Lemma~2.5]{Tal94} and the general case treated here is similar.

\begin{lemma} \label{lem:orlicz-upper}
Let $(E,\|\cdot\|_E)$ be a Banach space and $(\Omega,\mu)$ a probability space. For every $r\in(1,\infty)$ and $s\in(0,\infty)$, there exists $D=D(r,s)\in(0,\infty)$ such that every function $h:\Omega\to E$ satisfies
\begin{equation}
\big\|h\big\|_{L_r(\log L)^{-s}(\mu;E)} \leq D\ \frac{\|h\|_{L_r(\mu;E)}}{1+\log^{s/r}\big(\|h\|_{L_r(\mu;E)}/\|h\|_{L_1(\mu;E)}\big)}.
\end{equation}
\end{lemma}

\begin{proof}
Without loss of generality, we will again assume that $E=\C$ and $h\geq0$. We will prove that
\begin{equation*}
\int_\Omega \frac{h^r}{\log^s(e+h)}\diff \mu \geq1 \ \ \ \Longrightarrow \quad\|h\|_{L_r(\mu)}^r \geq \frac{1}{D^r} \big(1+\log^{s}\big(\|h\|_{L_r(\mu;E)}/\|h\|_{L_1(\mu;E)}\big)\big).
\end{equation*}
Let $a\in(0,\infty)$. We will distinguish two cases.

\smallskip

\noindent {\it Case 1.} Suppose that
\begin{equation*}
\int_{\{h\geq a\}} \frac{h^r}{\log^s(e+h)}\diff\mu \geq\frac{1}{2}.
\end{equation*}
Then,
\begin{equation} \label{eq:case1}
\int_\Omega h^r\diff\mu \geq \log^s(e+a) \int_{\{h\geq a\}} \frac{h^r}{\log^s(e+h)}\diff\mu \geq\frac{1}{2}\log^s(e+a).
\end{equation}

\noindent {\it Case 2.} Suppose that
\begin{equation*}
\int_{\{h\geq a\}} \frac{h^r}{\log^s(e+h)}\diff\mu <\frac{1}{2},
\end{equation*}
so that
\begin{equation*}
\int_{\{h< a\}} \frac{h^r}{\log^s(e+h)}\diff\mu \geq\frac{1}{2}.
\end{equation*}
Notice that on $\{h<a\}$, we have $h^r/\log^s(e+h) \leq a^{r-1} h$, which implies that $\|h\|_{L_1(\mu)} \geq 1/2a^{r-1}$. Hence, setting $b=\log\big(e\|h\|_{L_r(\mu)}/\|h\|_{L_1(\mu)}\big)$, we get
\begin{equation} \label{eq:case2}
b\leq \log(2ea^{r-1} \|h\|_{L_r(\mu)}\big) = (r-1)\log a +\log\big(2e\|h\|_{L_r(\mu)}\big).
\end{equation}

Now choose $a=\big(e\|h\|_{L_r(\mu)}/\|h\|_{L_1(\mu)}\big)^{1/r}$ so that $b=r\log a$. In Case 1,~\eqref{eq:case1} then implies that
\begin{equation*}
\|h\|_{L_r(\mu)}^r \geq \frac{1}{2}\log^s\big(e+\big(e\|h\|_{L_r(\mu)}/\|h\|_{L_1(\mu)}\big)^{1/r}\big) \asymp_{r,s}\big(1+\log^{s}\big(\|h\|_{L_r(\mu;E)}/\|h\|_{L_1(\mu;E)}\big)\big).
\end{equation*}
On the other hand since $b=r\log a$, in Case 2,~\eqref{eq:case2} gives
\begin{equation*}
\|h\|_{L_r(\mu)}^r \geq \frac{1}{(2e)^r} \frac{e\|h\|_{L_r(\mu)}}{\|h\|_{L_1(\mu)}} \gtrsim_{r,s} \big(1+\log^{s}\big(\|h\|_{L_r(\mu;E)}/\|h\|_{L_1(\mu;E)}\big)\big),
\end{equation*}
since $x \gtrsim_s1+ \log^sx$ for every $s,x\in(0,\infty)$. This completes the proof of the lemma.
\end{proof}
\section{Influence inequalities under Rademacher type} \label{sec:3}

In this section we shall present the proofs of Theorems~\ref{thm:useivv} and~\ref{thm:useivvorlicz} which rely on the novel approach introduced in the recent work~\cite{IVV20} of Ivanisvili, van Handel and Volberg. For $t\in(0,\infty)$, let $\xi(t)=\big(\xi_1(t),\ldots,\xi_n(t)\big)$ be a random vector on $\ms{C}_n$ whose coordinates are independent and identically distributed with distribution given by
\begin{equation}
\mb{P}\{\xi_i(t)=1\} = \frac{1+e^{-t}}{2} \quad \mbox{and} \quad\mb{P}\{\xi_i(t)=-1\}=\frac{1-e^{-t}}{2},
\end{equation}
for $i\in\{1,\ldots,n\}$. Moreover, consider the normalized vector $\delta(t)=(\delta_1(t),\ldots,\delta_n(t))$ with
\begin{equation}
\delta_i(t) \eqdef \frac{\xi_i(t)-\mb{E}\xi_i(t)}{\sqrt{\mathrm{Var}\xi_i(t)}} = \frac{\xi_i(t)-e^{-t}}{\sqrt{1-e^{-2t}}}.
\end{equation}
In the following statements, we will denote by $\e$ a random vector independent of $\xi(t)$, uniformly distributed on $\ms{C}_n$. We will need the following (straightforward) refinement of~\cite[Theorem~1.4]{IVV20}.

\begin{proposition} \label{prop:extendivv}
For every Banach space $(E,\|\cdot\|_E)$, $p\in[1,\infty)$, $n\in\N$ and $f:\ms{C}_n\to E$, we have
\begin{equation} \label{eq:propextendivv}
\forall \ t\geq0, \qquad \Big\|\frac{\partial}{\partial t} P_{t} f\Big\|_{L_p(\sigma_n;E)} \leq \frac{1}{\sqrt{e^{2t}-1}} \Big( \mb{E} \Big\| \sum_{i=1}^n \delta_i(t) \partial_i f(\e)\Big\|_{E}^p\Big)^{1/p},
\end{equation}
where the expectation on the right hand side is with respect to $\e$ and $\delta(t)$.
\end{proposition}

Let us mention here that we will apply the previous proposition to $P_tf$ instead of $f$, and use the semigroup property  $P_{2t}f= P_t(P_tf)$. This is more easily done after reformulating~\eqref{eq:propextendivv} with $\Delta P_t$ in place of $\tfrac{\partial}{\partial t} P_{t}$.  So, keeping the notation of Proposition \ref{prop:extendivv}, we have that 
\begin{equation}\label{eq:propwithPt}
\forall \ t\geq0, \qquad \|\Delta P_{2t} f\|_{L_p(\sigma_n;E)} \leq \frac{1}{\sqrt{e^{2t}-1}} \Big( \mb{E} \Big\| \sum_{i=1}^n \delta_i(t) \partial_i P_t f(\e)\Big\|_{E}^p\Big)^{1/p}.
\end{equation}
\begin{proof}[Proof of Proposition~\ref{prop:extendivv}]
The crucial observation of Ivanisvili, van Handel and Volberg is that one can write, for $x\in \ms{C}_n$,
\begin{equation}\label{eq:crucial-ivv}
\frac{\partial}{\partial t} P_{t} f (x) =  -\frac{1}{\sqrt{e^{2t}-1}} \mb{E}_{\xi(t)}\left[  \sum_{i=1}^n \delta_i(t) \partial_i f\big(x \xi(t)\big)\right]
\end{equation}
where $x\xi(t)$ denotes the point $(x_1\xi_1(t),\ldots,x_n\xi_n(t))$. This formula can be proved by writing 
$$\displaystyle P_tf(x) = \E f(x\xi(t)) = \sum_{\xi\in \ms{C}_n } \omega_t(\xi) f(x \xi),$$ 
where, for $\xi\in \ms{C}_n $,
$ \omega_t(\xi) = 2^{-n} \prod_{i=1}^{n} \big(1+ e^{-t} \xi_i\big)$; then we note that, with some abuse of notation (denoting  $\partial_{\xi_i} $ for the discrete derivative $\partial_i$ for functions of the variable $\xi\in \ms{C}_n$),
$$\frac{\partial}{\partial t}\omega_t (\xi) =  -\frac{e^{-t}}{1-e^{-2t}}\sum_{i=1}^{n} \partial_{\xi_i} \big[(\xi_i -e^{-t}) \omega_t (\xi)\big] 
.$$
Hence, using the integration by parts formula~\eqref{eq:ipp}  together with the fact that $\partial_{\xi_i} [f(x\xi)] = \partial_i f(x\xi)$, we get
$$\frac{\partial}{\partial t} P_{t} f (x) = - \frac{e^{-t}}{\sqrt{1-e^{-2t}}} \sum_{i=1}^n \sum_{\xi\in\ms{C}_n} \frac{\xi_i-e^{-t}}{\sqrt{1-e^{-2t}}} \omega_t(\xi) \partial_if(x\xi)= -\frac{1}{\sqrt{e^{2t}-1}} \mb{E}_{\xi(t)}\left[  \sum_{i=1}^n \delta_i(t) \partial_i f\big(x \xi(t)\big)\right]$$  
and this concludes the proof of \eqref{eq:crucial-ivv}. Alternatively, it suffices to readily check the validity of formula~\eqref{eq:crucial-ivv} in the case of the scalar-valued Walsh basis $w_J(x)=\prod_{j\in J} x_j $, where $J\subset \{1, \ldots, n\}$, for which $P_t w_J(x) = e^{-t|J|} w_J(x) $ and $\partial_i w_J(x) ={\bf 1}_{i\in J} \, w_J(x)$. 
 
 Therefore, using Jensen's inequality and \eqref{eq:crucial-ivv} we have
 \begin{eqnarray*}
\sqrt{e^{2t}-1} \Big\|\frac{\partial}{\partial t} P_{t} f\Big\|_{L_p(\sigma_n;E)}\!\!\!=  \Big(\E_\e \Big\|  \mb{E}_{\xi(t)}\,  \sum_{i=1}^n \delta_i(t) \partial_i f\big(\e \xi(t)\big)\Big\|_E^p\Big)^{1/p}  \!\!
 \le  \Big(\mb{E}\, \big\|  \sum_{i=1}^n \delta_i(t) \partial_i f\big(\e \xi(t)\big)\big\|_E^p\Big)^{1/p}.
 \end{eqnarray*}
We conclude by noting that the couple $(\e \xi(t), \xi(t))$ has the same law as the couple  $(\e, \xi(t))$.  This can be seen as a proxy of the rotational invariance of the Gaussian measure (compare with the proof of Proposition~\ref{prop:gaussian-bound} below).
\end{proof}

Theorems~\ref{thm:useivv} and~\ref{thm:useivvorlicz} are consequences  of the following lemma.
\begin{lemma} \label{lem:usetypeforivv}
Let $(E,\|\cdot\|_E)$ be a Banach space with Rademacher type 2. Then there exists a constant $K=K(E)\in(0,\infty)$ such that for every $\e\in(0,1)$ and $n\in\N$, every $f:\ms{C}_n\to E$ satisfies
\begin{equation}
\big\|f-\mb{E}_{\sigma_n}f\big\|^2_{L_2(\sigma_n;E)}  \leq \frac{K}{\e} \sum_{i=1}^n \int_0^\infty e^{-\e t} \|\partial_iP_t f\|_{L_2(\sigma_n;E)}^2 \ \frac{\diff t}{t^\e},
\end{equation}
\end{lemma}

\begin{proof}
We will apply Proposition~\ref{prop:extendivv} to $P_tf$ instead of $f$. We have that
\begin{multline} \label{eq:usetypeforivv1}
\big\|f-\mb{E}_{\sigma_n}f\big\|_{L_2(\sigma_n;E)} = \Big\|  \int_0^\infty \Delta P_t f \diff t\Big\|_{L_2(\sigma_n;E)}  = 2\Big\|\int_0^\infty\Delta P_{2t}f\diff t\Big\|_{L_2(\sigma_n;E)} \\ 
\leq 2 \int_0^\infty \|\Delta P_{2t}f\|_{L_2(\sigma_n;E)} \diff t
  \stackrel{\eqref{eq:propwithPt}}{\leq} 2 \int_0^\infty \Big( \mb{E} \Big\| \sum_{i=1}^n \delta_i(t) \partial_i P_t f(\e)\Big\|_{E}^2\Big)^{1/2}\frac{\diff t}{\sqrt{e^{2t}-1}}.
\end{multline}
Suppose now that $E$ has Rademacher type 2 with constant $T$. Then for $\e\in(0,1)$, by~\eqref{eq:usetypeforivv1} and the Rademacher type condition for centered random variables~\cite[Proposition~9.11]{LT91}, we have
\begin{equation} \label{eq:65}
\begin{split}
\big\|f-\mb{E}_{\sigma_n}f\big\|_{L_2(\sigma_n;E)} & \leq 4 T \int_0^\infty \Big(\sum_{i=1}^n \big\|\partial_i P_tf\big\|_{L_2(\sigma_n;E)}^2\Big)^{1/2} \frac{\diff t}{\sqrt{e^{2t}-1}}
\\ & \leq 4T \Big( \int_0^\infty \sum_{i=1}^n \big\|\partial_i P_tf\big\|_{L_2(\sigma_n;E)}^2 \frac{\diff t}{(e^{2t}-1)^\e}\Big)^{1/2}\Big(\int_0^\infty \frac{\diff t}{(e^{2t}-1)^{1-\e}}\Big)^{1/2},
\end{split}
\end{equation}
where in the second line we used the Cauchy--Schwarz inequality. Therefore, since the integral $\int_0^\infty \frac{\diff t}{(e^{2t}-1)^{1-\e}} \asymp \tfrac{1}{\e}$ as $\e\to0^+$, we deduce that there exists a universal constant $C\in(0,\infty)$ with
\begin{equation*}
\big\|f-\mb{E}_{\sigma_n}f\big\|^2_{L_2(\sigma_n;E)}  \leq \frac{C\cdot T^2}{\e} \sum_{i=1}^n \int_0^\infty\|\partial_i P_tf\|_{L_2(\sigma_n;E)}^2 \ \frac{\diff t}{(e^{2t}-1)^\e},
\end{equation*}
and the conclusion follows readily since $e^{2t}-1\geq te^t$ for every $t\geq0$.
\end{proof}

\begin{proof} [Proof of Theorems~\ref{thm:useivv} and~\ref{thm:useivvorlicz}]
By Bonami's hypercontractive inequalities~\cite{Bon70}, since the semigroup commutes with partial derivatives, we get that for every $t\geq0$ and $i\in\{1,\ldots,n\}$,
\begin{equation} \label{eq:first-use-bonami}
 \|\partial_iP_{t} f\|_{L_2(\sigma_n;E)}= \|P_{t}\partial_i f\|_{L_2(\sigma_n;E)} \leq \|\partial_i f\|_{L_{1+e^{-2t}}(\sigma_n;E)}.
\end{equation}
Therefore, the conclusion of Theorem~\ref{thm:useivvorlicz} follows by combining Lemma~\ref{lem:usetypeforivv},~\eqref{eq:first-use-bonami} and Lemma~\ref{lem:orlicz-withpower}. Moreover, in view of Lemma~\ref{lem:orlicz-upper}, Theorem~\ref{thm:useivvorlicz} readily implies~\eqref{eq:thmuseivv}. In order to prove~\eqref{eq:thmuseivvloglog}, one can just apply~\eqref{eq:thmuseivv} for $\e\asymp\sigma(f)^{-1}$.
\end{proof}

\begin{remark} \label{rem:referee}
It was pointed out to us by an anonymous referee that plugging in the standard application \eqref{eq:log-nopower2} of Hölder's inequality along with hypercontractivity to bound the middle term of \eqref{eq:65} cannot remove the dependence on $\varepsilon$ in inequality \eqref{eq:thmuseivv}. Indeed, by hypercontractivity and Hölder's inequality, we have
$$\int_0^\infty \Big(\sum_{i=1}^n \big\|\partial_i P_tf\big\|_{L_2(\sigma_n;E)}^2\Big)^{1/2} \frac{\diff t}{\sqrt{e^{2t}-1}} \leq \int_0^1 \Big( \sum_{i=1}^n a_i^{\frac{1-u^2}{1+u^2}} b_i^{\frac{2u^2}{1+u^2}} \Big)^{1/2} \frac{\diff u}{\sqrt{1-u^2}},$$
where $a_i = \|\partial_i f\|_{L_1(\sigma_n;E)}^2$ and $b_i=\|\partial_i f\|_{L_2(\sigma_n;E)}^2$. Suppose, for contradiction, that for every $n\geq1$ and every $0\leq a_i\leq b_i$ where $i\in\{1,\ldots,n\}$, we have
$$\int_0^1 \Big( \sum_{i=1}^n a_i^{\frac{1-u^2}{1+u^2}} b_i^{\frac{2u^2}{1+u^2}} \Big)^{1/2} \frac{\diff u}{\sqrt{1-u^2}} \lesssim \Big(\sum_{i=1}^n \frac{b_i}{1+\log(b_i/a_i)} \Big)^{1/2}.$$
Equivalently, we have
\begin{equation} \label{eq:ref1}
\int_0^1 \Big( \sum_{i=1}^n p_i \exp\big(-\tfrac{1-u^2}{1+u^2} x_i\big)\cdot (1+x_i)\Big)^{1/2}\frac{\diff u}{\sqrt{1-u^2}} \lesssim 1
\end{equation}
where $x_i = \log(b_i/a_i)\geq0$ and $\big( \sum_{k=1}^n \tfrac{b_k}{1+\log(b_k/a_k)} \big) p_i = \tfrac{b_i}{1+\log(b_i/a_i)}$. The parameters $n\geq1$, $x_i\geq0$ and the weights $p_i$ are all arbitrary, thus we conclude from \eqref{eq:ref1} that for every positive random variable $X$, the inequality
\begin{equation} \label{eq:ref2}
\int_0^1 \sqrt{ \mb{E} \big[ \exp\big(-\tfrac{1-u^2}{1+u^2} X\big) \cdot (1+X)\big]} \, \frac{\diff u}{\sqrt{1-u^2}} \lesssim1\end{equation}
holds true. To reach a contradiction, consider a discrete random variable $X\geq0$ such that
\begin{equation} \label{eq:div-prob}
\sum_{k\geq0} \sqrt{\mb{P}\big\{ 1+X \in[2^k, 2^{k+1}) \big\}} = \infty
\end{equation}
and notice that
\begin{equation*}
\begin{split}
\int_0^1 \sqrt{ \mb{E} \big[ \exp\big(-\tfrac{1-u^2}{1+u^2} X\big) \cdot (1+X)\big]}& \, \frac{\diff u}{\sqrt{1-u^2}} \geq \frac{1}{\sqrt{2}} \int_0^1 \sqrt{\mb{E}\big[\exp\big(-v(1+X) \big)\cdot (1+X) \big]} \,\frac{\diff v}{\sqrt{v}}
\\ & >\frac{1}{\sqrt{2}}  \int_0^1 \sqrt{\sum_{k=0}^\infty \exp\big(-v2^{k+1}\big)\cdot 2^k \cdot \mb{P}\big\{ 1+X \in[2^k, 2^{k+1}) \big\}} \,\frac{\diff v}{\sqrt{v}} 
\\ &  \geq\frac{1}{2\sqrt{2}} \sum_{\ell=0}^\infty 2^{-\frac{\ell}{2}}\sqrt{\sum_{k=0}^\infty \exp\big(-2^{k-\ell}\big)\cdot 2^k \cdot \mb{P}\big\{ 1+X \in[2^k, 2^{k+1}) \big\}} 
\\ & \geq \frac{1}{2\sqrt{2e}} \sum_{\ell=0}^\infty \sqrt{\mb{P}\big\{ 1+X \in[2^\ell, 2^{\ell+1}) \big\}} = \infty,
\end{split}
\end{equation*}
where in the last inequality we bounded the inner sum by the $k=\ell$ term. This contradicts \eqref{eq:ref2}.
\end{remark}

\begin{remark}
A combination of Proposition~\ref{prop:extendivv} and Lemma~\ref{lem:log-nopower} implies a different Talagrand-type strengthening of the vector-valued discrete Poincar\'e inequality~\eqref{eq:vectorpoincare} for spaces of Rademacher type 2, which is weaker than~\eqref{eq:vectortalagrand} (see also~\cite[Theorem~5.4]{Cha14} for a similar scalar-valued inequality). For a function $f:\ms{C}_n\to E$, we will use the notation ${\bf D}f:\ms{C}_n\to E^n$ for the gradient vector
\begin{equation*}
{\bf D}f \eqdef \big(\partial_1f,\ldots,\partial_nf\big).
\end{equation*}
Then, the first inequality in~\eqref{eq:65} can be rewritten as
\begin{equation*}
\big\|f-\mb{E}_{\sigma_n}f\big\|_{L_2(\sigma_n;E)} \lesssim_E \int_0^\infty \Big(\sum_{i=1}^n \big\|\partial_i P_tf\big\|_{L_2(\sigma_n;E)}^2\Big)^{1/2} \frac{\diff t}{\sqrt{e^{2t}-1}} =\int_0^\infty \big\|P_t{\bf D}f\big\|_{L_2(\sigma_n;\ell_2^n(E))} \frac{\diff t}{\sqrt{e^{2t}-1}}.
\end{equation*}
Now, by the hypercontractivity of $\{P_t\}_{t\geq0}$, we have
\begin{equation*}
\big\|P_{t}{\bf D}f\big\|_{L_2(\sigma_n;\ell_2^n(E))}\leq \big\|{\bf D}f\big\|_{L_{1+e^{-2t}}(\sigma_n;\ell_2^n(E))}.
\end{equation*}
Therefore, combining the last two inequalities, we get
\begin{equation*}
\big\|f-\mb{E}_{\sigma_n}f\big\|_{L_2(\sigma_n;E)} \lesssim_E \int_0^\infty \big\|{\bf D}f\big\|_{L_{1+e^{-2t}}(\sigma_n;\ell_2^n(E))} \frac{\diff t}{\sqrt{e^{2t}-1}} \lesssim \int_0^\infty e^{-t/2} \big\|{\bf D}f\big\|_{L_{1+e^{-2t}}(\sigma_n;\ell_2^n(E))} \frac{\diff t}{\sqrt{t}}
\end{equation*}
and Lemma~\ref{lem:log-nopower} then implies that
\begin{equation} \label{eq:talagrand-including-type}
\big\|f-\mb{E}_{\sigma_n}f\big\|_{L_2(\sigma_n;E)} \lesssim_E \frac{\big\|{\bf D}f\big\|_{L_2(\sigma_n;\ell_2^n(E))}}{1+\sqrt{\log\big(\big\|{\bf D}f\big\|_{L_2(\sigma_n;\ell_2^n(E))}/\big\|{\bf D}f\big\|_{L_{1}(\sigma_n;\ell_2^n(E))}\big)}}.
\end{equation}
The argument above shows that spaces of Rademacher type 2 satisfy~\eqref{eq:talagrand-including-type} and the reverse implication is clear by choosing a function of the form $f(\e)=\sum_{i=1}^n\e_i x_i$. When $E=\C$, this coincides with~\eqref{eq:thml1lpscalar} where $p=2$ (see also Remark~\ref{rem:l1l2scalar} below for comparison with~\eqref{eq:talagrand}).
\end{remark}
\section{Influence inequalities under martingale type} \label{sec:4}

In this section, we shall present two proofs of Theorems~\ref{thm:useeldan} and~\ref{thm:useeldanorlicz}, one probabilistic and one Fourier analytic. As a warmup, we present a simple proof of Talagrand's inequality in Gauss space for functions with values in a space of martingale type 2 using a classical stochastic representation for the variance. The scalar-valued case of this inequality was shown in~\cite{CL12} via semigroup methods which do not seem to be adaptable to the case of vector-valued functions (see Section \ref{subsec:lps} for a harmonic analytic variant). We will denote by $\gamma_n$ the standard Gaussian measure on $\R^n$, i.e.~ the measure $\diff\gamma_n(x) = \tfrac{\exp(-\|x\|_2^2/2)}{(2\pi)^{n/2}}\diff x$, where $\|\cdot\|_2$ denotes the usual Euclidean norm on $\R^n$.


\subsection{A simple stochastic proof in Gauss space} We will denote by $\{U_t\}_{t\geq0}$ the Ornstein--Uhlenbeck semigroup on $\R^n$, whose action on an integrable function $f:\R^n\to E$, where $(E,\|\cdot\|_E)$ is a Banach space, is given by the Mehler formula
\begin{equation} \label{eq:mehler}
\forall \ t\geq0 \mbox{ and } x\in\R^n, \quad\ U_tf(x) = \int_{\R^n}f\big(e^{-t}x+\sqrt{1-e^{-2t}}y\big) \diff\gamma_n(y).
\end{equation}
Let $\{X_t\}_{t\geq0}$ be an Ornstein--Uhlenbeck process, i.e. a stochastic process of the form $X_t = e^{-t}X_0+e^{-t} B_{e^{2t}-1}$, where $\{B_t\}_{t\geq0}$ is a standard Brownian motion and $X_0$ is a standard Gaussian random vector, independent of $\{B_t\}_{t\geq0}$. We will use the following well-known consequence of the Clark--Ocone formula (see~\cite{CHL97} for a proof and further applications in functional inequalities).

\begin{lemma} \label{lem:stochastic}
Let $(E,\|\cdot\|_E)$ be a Banach space. For every smooth function $f:\R^n\to E$, we have
\begin{equation} \label{eq:chl}
\forall \ s>0, \quad\ f(X_s) - U_sf(X_0) = \int_0^s \nabla \big(U_{s-t}f\big)(X_t)  \cdot \diff B_t.
\end{equation}
\end{lemma}

We will also need the following one-sided version of the Itô isometry for 2-smooth spaces, which is essentially due to Dettweiler~\cite{Det91}. We include the crux of the (simple) proof for completeness.

\begin{proposition} \label{prop:stochastic2}
Let $(E,\|\cdot\|_E)$ be a Banach space of martingale type 2. Then, there exists $M\in(0,\infty)$ such that for every $n\in\N$, if $\{B_t\}_{t\geq0}$ is a standard Brownian motion on $\R^n$ and $\{Y_t\}_{t\geq0}$ is an $E^n$-valued square integrable stochastic process adapted to the filtration $\{\ms{F}_t\}_{t\geq0}$ of $\{B_t\}_{t\geq0}$, then
\begin{equation}
\mb{E}\Big\|\int_0^\infty Y_t \cdot \diff B_t \Big\|_E^2 \leq M^2 \int_0^\infty \mb{E}\Big\|\sum_{i=1}^n G(i) Y_t(i) \Big\|_E^2\diff t,
\end{equation}
where $G=(G(1),\ldots,G(n))$ is a standard Gaussian random vector on $\R^n$, independent of $\{\ms{F}_t\}_{t\geq0}$.
\end{proposition}

\begin{proof}
We shall assume that $\{Y_t\}_{t\geq0}$ is a simple process of the form
\begin{equation*}
\forall \ i\in\{1,\ldots,n\}, \qquad Y_t(i) = \sum_{k=1}^N \alpha_{t_k}(i)\cdot{\bf 1}_{(t_k, t_{k+1}]},
\end{equation*}
where $0=t_1<t_2<\ldots<t_{N+1}$ and each $\alpha_{t_k}(i)$ is an $\ms{F}_{t_k}-$measurable random variable. The general case will follow by standard approximation arguments. By definition,
\begin{equation*}
\int_0^\infty Y_t \cdot \diff B_t = \sum_{k=1}^N \sum_{i=1}^n \alpha_{t_k}(i) \cdot \big( B_{t_{k+1}}(i)-B_{t_k}(i)\big)
\end{equation*}
and $\big\{\sum_{i=1}^n\alpha_{t_k}(i) ( B_{t_{k+1}}(i)-B_{t_k}(i))\big\}_{k=1}^N$  is a martingale difference sequence, therefore if $M$ is the martingale type 2 constant of $E$,
\begin{equation} \label{eq:stochastic-use-mtype}
\mb{E}\Big\|\int_0^\infty Y_t \cdot \diff B_t \Big\|_E^2 \leq M^2 \sum_{k=1}^N\mb{E}\Big\| \sum_{i=1}^n \alpha_{t_k}(i) \cdot \big( B_{t_{k+1}}(i)-B_{t_k}(i)\big) \Big\|_E^2.
\end{equation}
Now, for a fixed $k$, $\big(B_{t_{k+1}}(i)-B_{t_k}(i)\big)_{i=1}^n$ conditioned of $\ms{F}_{t_k}$ is equidistributed to a Gaussian random vector with covariance matrix $(t_{k+1}-t_k)\cdot \msf{Id}_n$. Therefore,
\begin{equation} \label{eq:expect-bm}
\mb{E}\left[ \Big\| \sum_{i=1}^n \alpha_{t_k}(i) \cdot \big( B_{t_{k+1}}(i)-B_{t_k}(i)\big) \Big\|_E^2 \ \Big| \ \ms{F}_{t_k}\right]= \big(t_{k+1}-t_k\big) \mb{E}\left[ \Big\|\sum_{i=1}^n G(i) \alpha_{t_k}(i)\Big\|_E^2 \ \Big| \ \ms{F}_{t_k}\right],
\end{equation}
where $G=(G(1),\ldots,G(n))$ is a standard Gaussian random vector, independent of $\{\ms{F}_t\}_{t\geq0}$. Hence, after taking expectation in~\eqref{eq:expect-bm} and summing over $k$,~\eqref{eq:stochastic-use-mtype} becomes
\begin{equation*}
\begin{split}
\mb{E}\Big\|\int_0^\infty Y_t \cdot \diff B_t \Big\|_E^2 \leq M^2 \sum_{k=1}^N \big(t_{k+1}-t_k\big) \mb{E}\Big\|\sum_{i=1}^n G(i) \alpha_{t_k}(i) \Big\|_E^2 = M^2 \int_0^\infty \mb{E}\Big\|\sum_{i=1}^n G(i) Y_t(i) \Big\|_E^2\diff t,
\end{split}
\end{equation*}
thus completing the proof of this simple fact.
\end{proof}

We are now well-equipped to prove the following result.

\begin{theorem}  \label{thm:usechl}
Let $(E,\|\cdot\|_E)$ be a Banach space with martingale type 2. Then, there exists $C=C(E)\in(0,\infty)$ such that for every $n\in\N$, every smooth function $f:\R^n\to E$ satisfies
\begin{equation} \label{eq:thmusechl}
\big\| f-\mb{E}_{\gamma_n}f\big\|_{L_2(\gamma_n;E)}^2 \leq C \sum_{i=1}^n  \|\partial_if\|_{L_2(\log L)^{-1}(\gamma_n;E)}^2.
\end{equation}
\end{theorem}

\begin{proof}
If $E$ has martingale type 2 with constant $M$, then Lemma~\ref{lem:stochastic} and Proposition~\ref{prop:stochastic2} imply that
\begin{equation*}
\forall \ s>0, \ \ \mb{E}\big[\big\|f(X_s)-U_sf(X_0)\big\|_E^2 \ \big| \ X_0\big]\leq M^2 \int_0^s \mb{E} \left[\Big\|\sum_{i=1}^n G(i) \partial_iU_{s-t}f(X_t) \Big\|_E^2 \ \Big| \ X_0 \right]\diff t.
\end{equation*}
Thus, applying the Rademacher type 2 condition for Gaussian variables, we deduce that
\begin{equation} \label{eq:gaussian-applied-types}
\forall \ s>0, \quad\mb{E}\big[ \big\|f(X_s)-U_sf(X_0)\big\|_E^2 \ \big| \ X_0\big]\leq M^2T^2 \int_0^s  \sum_{i=1}^n \mb{E} \left[ \big\|\partial_iU_{s-t}f(X_t) \big\|_E^2 \ \Big| \ X_0 \right] \diff t,
\end{equation}
where $T$ is the Rademacher type 2 constant of $E$. Now, integrating~\eqref{eq:gaussian-applied-types} with respect to the standard Gaussian random vector $X_0$ and using the stationarity of the Ornstein--Uhlenbeck process $\{X_t\}_{t\geq0}$ along with Nelson's hypercontractive inequalities~\cite{Nel66, Nel73}, we derive
\begin{equation} \label{eq:gaussian-applied-types2}
\begin{split}
\forall & \ s>0, \qquad \mb{E}\big\|f(X_s)-U_sf(X_0)\big\|_E^2\leq M^2T^2 \sum_{i=1}^n \int_0^s  \big\|\partial_iU_{s-t}f\big\|_{L_2(\gamma_n;E)}^2 \diff t
\\ & = M^2T^2 \sum_{i=1}^n \int_0^s e^{-2(s-t)} \big\|U_{s-t}\partial_if\big\|_{L_2(\gamma_n;E)}^2 \diff t \leq M^2T^2 \sum_{i=1}^n \int_0^s e^{-2t} \|\partial_if\|^2_{L_{1+e^{-2t}}(\gamma_n;E)} \diff t,
\end{split},
\end{equation}
where the equality follows from the standard commutation relation $\partial_i U_{s-t} f = e^{-(s-t)}U_{s-t}\partial_if$. Since for every $i\in\{1,\ldots,n\}$ the correlation $\mb{E} X_0(i)X_s(i) = e^{-s}$, taking $s\to\infty$ in~\eqref{eq:gaussian-applied-types2} we get
\begin{equation*}
\big\|f-\mb{E}_{\gamma_n}f\big\|_{L_2(\gamma_n;E)}^2 \leq M^2T^2\int_0^\infty e^{-2t} \|f\|^2_{L_{1+e^{-2t}}(\gamma_n;E)} \diff t
\end{equation*}
and the conclusion follows by Lemma~\ref{lem:orlicz-withpower}.
\end{proof}

\subsection{A proof of Theorems~\ref{thm:useeldan} and~\ref{thm:useeldanorlicz} via the Eldan--Gross process} \label{subsec:eldan} In a recent paper, Eldan and Gross~\cite{EG19} constructed a clever stochastic process on the cube which resembles the behavior of Brownian motion on $\R^n$ and used it to prove several important inequalities relating the variance and influences of Boolean functions. We shall briefly describe their construction.

Let $\{B_t\}_{t\geq0}=\big\{\big(B_t(1),\ldots,B_t(n)\big)\big\}_{t\geq0}$ be a standard Brownian motion on $\R^n$. For every $i\in\{1,\ldots,n\}$ and $t\geq0$, consider the stopping time $\tau_t(i)$ given by
\begin{equation*}
\tau_t(i) \eqdef\inf\big\{s\geq0: \ |B_s(i)|>t\big\}.
\end{equation*}
and then, let $X_t(i) \eqdef B_{\tau_t(i)}(i)$. Then, the jump process $\{X_t\}_{t\geq0}\eqdef\big\{\big(X_t(1),\ldots,X_t(n)\big)\big\}_{t\geq0}$ satisfies the following properties (see~\cite[Section~3]{EG19} for detailed proofs):
\begin{enumerate} [\bf 1.]
\item For every $t\geq0$ and $i\in\{1,\ldots,n\}$, $\big|X_t(i)\big|=t$ almost surely and in fact $X_t\sim\mathrm{Unif}\{-t,t\}^n$.
\item The process $\{X_t\}_{t\geq0}$ is a martingale.
\item For every coordinate $i\in\{1,\ldots,n\}$, the jump probabilities of $\{X_t(i)\}_{t\geq0}$ are
\begin{equation} \label{eq:jump-prob}
\forall \ t,h>0, \qquad  \mb{P}\big\{ \mathrm{sign}X_{t+h}(i) \neq \mathrm{sign}X_t(i)\big\} = \frac{h}{2(t+h)}.
\end{equation}
\end{enumerate}

\begin{proof} [Proof of Theorems~\ref{thm:useeldan} and~\ref{thm:useeldanorlicz}]
Fix a function $f:\ms{C}_n\to E$ and recall (see, e.g.,~\cite{O'Do14}) that there exists a unique multilinear polynomial on $\R^n$, which coincides with $f$ on $\ms{C}_n$, i.e. we can write
\begin{equation} \label{eq:walsh-expansion}
\forall \ \e\in\ms{C}_n, \qquad f(\e) = \sum_{A\subseteq\{1,\ldots,n\}} \widehat{f}(A) \prod_{i\in A}\e_i,
\end{equation}
for some coefficients $\widehat{f}(A)\in E$. By abuse of notation, we will also denote by $f$ that unique multilinear extension on $\R^n$. Since $f$ is a multilinear polynomial and $\{X_t\}_{t\geq0}$ is a martingale with independent coordinates, it follows that the process $\{f(X_t)\}_{t\geq0}$ is itself a martingale.

Fix some large $N\in\N$ and for $k\in\{0,1,\ldots,N\}$, let $t_k=\tfrac{k}{N}$ and $M_k=f(X_{t_k})$. Since $E$ has martingale type 2, there exists $M=M(E)\in(0,\infty)$ such that
\begin{equation} \label{use-mart-type}
\big\|f-\mb{E}_{\sigma_n}f\big\|_{L_2(\sigma_n;E)}^2 = \mb{E}\|M_N-M_0\|_E^2 \leq M^2\sum_{k=1}^N \mb{E}\|M_{k}-M_{k-1}\|_E^2.
\end{equation}
Now, for a fixed $k\in\{1,\ldots,N\}$, since $M_k-M_{k-1}=f(X_{t_k})-f(X_{t_{k-1}})$, Taylor's formula gives
\begin{equation} \label{eq:taylor}
\begin{split}
M_k-M_{k-1} = & \sum_{i=1}^n\big( X_{t_k}(i)-X_{t_{k-1}}(i)\big) \cdot \frac{\partial f}{\partial x_i}(X_{t_{k-1}})+R_k(f),
\end{split}
\end{equation}
where $\frac{\partial f}{\partial x_i}$ are the usual partial derivatives of $f$ on $\R^n$ and the remainder $R_k(f)$ satisfies
\begin{equation} \label{eq:remainder}
\|R_k(f)\|_E\leq \frac{1}{2}\sum_{i,j=1}^n \left\| \frac{\partial^2 f}{\partial x_i\partial x_j} \right\|_{L_\infty([-1,1]^n;E)} \big| X_{t_k}(i)-X_{t_{k-1}}(i)\big| \cdot  \big| X_{t_k}(j)-X_{t_{k-1}}(j)\big|.
\end{equation}
However, since $f$ is a multilinear polynomial, all second derivatives of the form $\partial^2 f/\partial x_i^2$ vanish and~\eqref{eq:remainder} implies that
\begin{equation} \label{eq:remainder2}
\|R_k(f)\|_E\leq K(f) \cdot \sum_{\substack{i,j=1 \\ i\neq j}}^n \big| X_{t_k}(i)-X_{t_{k-1}}(i)\big| \cdot  \big| X_{t_k}(j)-X_{t_{k-1}}(j)\big|,
\end{equation}
for some $K(f)\in(0,\infty)$, so that
\begin{equation} \label{eq:remainder3}
\mb{E} \|R_k(f)\|_E^2 \leq n^2 K(f)^2 \cdot \sum_{\substack{i,j=1 \\ i\neq j}}^n \mb{E} \big| X_{t_k}(i)-X_{t_{k-1}}(i)\big|^2 \cdot  \mb{E}\big| X_{t_k}(j)-X_{t_{k-1}}(j)\big|^2.
\end{equation}
The fact that only $i\neq j$ enters the sum will be crucial below to ensure that the error tends to zero as $N\to+\infty$ after summing over $k$. Now, by~\eqref{eq:jump-prob}, we have
\begin{equation*}
\mathrm{sign}(X_{t_{k-1}}(i))\cdot\big(X_{t_k}(i)-X_{t_{k-1}}(i)\big)= \begin{cases} -\frac{2k-1}{N}, & \mbox{with probability } \frac{1}{2k} \\ \frac{1}{N}, & \mbox{with probability } \frac{2k-1}{2k} \end{cases},
\end{equation*}
so the conditional second moment of the increments is
\begin{equation} \label{eq:cond-L2}
\mb{E}\big[ \big|X_{t_k}(i)-X_{t_{k-1}}(i)\big|^2 \big| X_{t_{k-1}}(i)\big] = \frac{1}{2k} \Big( \frac{2k-1}{N}\big)^2 + \frac{2k-1}{2k} \frac{1}{N^2} = \frac{2k-1}{N^2}.
\end{equation}
By the tower property of conditional expectation, the estimate~\eqref{eq:remainder3} can finally be written as
\begin{equation} \label{eq:remainder4}
\mb{E} \|R_k(f)\|_E^2 \lesssim \frac{k^2 n^4 K(f)^2}{N^4}
\end{equation}
and thus~\eqref{eq:taylor} implies that
\begin{equation} \label{eq:taylor2}
\mb{E}\|M_k-M_{k-1}\|_E^2 \lesssim \mb{E}\left\|\sum_{i=1}^n \big( X_{t_k}(i)-X_{t_{k-1}}(i)\big) \cdot \frac{\partial f}{\partial x_i}(X_{t_{k-1}})\right\|_E^2 + \frac{k^2 n^4 K(f)^2}{N^4}.
\end{equation}
Since $\{X_t\}_{t\geq0}$ is a martingale, the sequence $(X_{t_k}(i)-X_{t_{k-1}}(i))_{i=1}^n$ is a sequence of independent centered random variables, when conditioned on $\{X_s\}_{s\leq t_{k-1}}$. Therefore, applying the Rademacher type condition for centered random variables~\cite[Proposition~9.11]{LT91} and~\eqref{eq:cond-L2}, we deduce that
\begin{equation} \label{eq:use-type-in-eldan}
\mb{E}\left[\left\|\sum_{i=1}^n \big( X_{t_k}(i)-X_{t_{k-1}}(i)\big) \cdot \frac{\partial f}{\partial x_i}(X_{t_{k-1}})\right\|_E^2 \  \bigg| \ \{X_s\}_{s\leq t_{k-1}}\right] \lesssim \frac{kT^2}{N^2} \sum_{i=1}^n \Big\|  \frac{\partial f}{\partial x_i}(X_{t_{k-1}})\Big\|_E^2,
\end{equation}
where $T$ is the type 2 constant of $E$. By the tower property of conditional expectation,~\eqref{eq:taylor2} combined with~\eqref{eq:use-type-in-eldan} gives
\begin{equation} \label{eq:taylor3}
\mb{E}\|M_k-M_{k-1}\|_E^2 \lesssim \frac{kT^2}{N^2} \sum_{i=1}^n \mb{E}\Big\|  \frac{\partial f}{\partial x_i}(X_{t_{k-1}})\Big\|_E^2 + \frac{k^2 n^4 K(f)^2}{N^4}.
\end{equation}
Now, summing over $k\in\{1,\ldots,N\}$ and using~\eqref{use-mart-type}, we get
\begin{equation} \label{eq:taylor4}
\big\|f-\mb{E}_{\sigma_n}f\big\|_{L_2(\sigma_n;E)}^2 \lesssim M^2T^2 \sum_{i=1}^n \frac{1}{N}\sum_{k=1}^N\frac{k}{N}\mb{E}\Big\|  \frac{\partial f}{\partial x_i}(X_{t_{k-1}})\Big\|_E^2 + \frac{n^4 K(f)^2M^2}{N},
\end{equation}
which as $N\to\infty$ becomes
\begin{equation} \label{eq:eldan-almost-done}
\big\|f-\mb{E}_{\sigma_n}f\big\|_{L_2(\sigma_n;E)}^2 \lesssim M^2T^2 \sum_{i=1}^n \int_0^1 t\mb{E}\Big\|  \frac{\partial f}{\partial x_i}(X_{t})\Big\|_E^2\diff t.
\end{equation}
Since $X_t$ is uniformly distributed on $\{-t,t\}^n$, the random variable $\frac{\partial f}{\partial x_i}(X_t)$ satisfies
\begin{equation} \label{eq:equal-dist}
\frac{\partial f}{\partial x_i}(X_t) =\sum_{\substack{A\subseteq\{1,\ldots,n\} \\ i\in A}} \widehat{f}(A) \prod_{j\in A\setminus\{i\}} X_t(j)\ \sim\!\!\! \sum_{\substack{A\subseteq\{1,\ldots,n\} \\ i\in A}} t^{|A|-1} \widehat{f}(A) \prod_{j\in A\setminus\{i\}} \e_j = P_{\log(1/t)} \frac{\partial f}{\partial x_i}(\e),
\end{equation}
where $\sim$ denotes equality in distribution, $\e$ is uniformly distributed on $\ms{C}_n$ and the last equality follows, e.g., by~\cite[Proposition~2.47]{O'Do14}. Therefore, by~\eqref{eq:equal-dist} and the change of variables $u=\log(1/t)$, we can rewrite ~\eqref{eq:eldan-almost-done} as
\begin{equation} \label{eq:eldan-almost-done2}
\begin{split}
\big\|f-\mb{E}_{\sigma_n}f\big\|_{L_2(\sigma_n;E)}^2 \lesssim M^2T^2 \sum_{i=1}^n \int_0^\infty &e^{-2u} \Big\|  P_u\frac{\partial f}{\partial x_i}\Big\|_{L_2(\sigma_n;E)}^2\diff u.
\end{split}
\end{equation}
In the scalar-valued case, formula~\eqref{eq:eldan-almost-done}  is then an equality with $M^2T^2=1$ and appears in~\cite{EG19}. However, in this case,  its equivalent form~\eqref{eq:eldan-almost-done2} can also be proved by elementary semigroup arguments as in~\cite{CL12} which we can follow to conclude the proof. 
Using hypercontractivity~\cite{Bon70} and~\eqref{eq:eldan-almost-done2}, we get
\begin{equation*}
\big\|f-\mb{E}_{\sigma_n}f\big\|_{L_2(\sigma_n;E)}^2 \lesssim M^2 T^2 \sum_{i=1}^n \int_0^\infty e^{-2u} \Big\| \frac{\partial f}{\partial x_i}\Big\|_{L_{1+e^{-2u}}(\sigma_n;E)}^2\diff u.
\end{equation*}
The conclusions of Theorems~\ref{thm:useeldan} and~\ref{thm:useeldanorlicz} now follow from~\eqref{eq:eldan-almost-done2} combined with Lemmas~\ref{lem:orlicz-withpower} and~\ref{lem:orlicz-upper} since for every $i\in\{1,\ldots,n\}$, we have $\frac{\partial f}{\partial x_i} (\e) = \e_i \partial_if(\e)$ for every $\e\in\ms{C}_n$.
\end{proof}

\subsection{A proof of Theorems~\ref{thm:useeldan} and~\ref{thm:useeldanorlicz} by Littlewood--Paley--Stein theory} \label{subsec:lps} We shall now present a second, more analytic proof of Theorems~\ref{thm:useeldan} and~\ref{thm:useeldanorlicz}. The main tool for this proof, is a deep vector-valued Littlewood--Paley--Stein inequality (see~\cite{Ste70}) due to Xu~\cite{Xu20}, which is the culmination of the series of works~\cite{Xu98, MTX06} (see also~\cite{Hyt07} for some similar inequalities for UMD targets). We will need the following statement which is a special case of~\cite[Theorem~2]{Xu20}.

\begin{theorem} [Xu]
Let $(E,\|\cdot\|_E)$ be a Banach space of martingale type 2. Then, there exists a constant $C=C(E)\in(0,\infty)$ such that for symmetric diffusion semigroup $\{T_t\}_{t>0}$ on a probability space $(\Omega,\mu)$, every function $f:\Omega\to E$ satisfies
\begin{equation} \label{eq:xu}
\big\|f-\mb{E}_{\mu}f\big\|_{L_2(\mu;E)}^2 \leq C^2 \int_0^\infty \|t\partial_t T_tf\|_{L_2(\mu;E)}^2 \frac{\diff t}{t}.
\end{equation}
\end{theorem}

\begin{proof} [Second proof of Theorems~\ref{thm:useeldan} and~\ref{thm:useeldanorlicz}]
Since $E$ has martingale type 2, there exists $T\in(0,\infty)$ such that $E$ also has Rademacher type 2 with constant $T$. Then, applying Proposition~\ref{prop:extendivv} to $P_tf$ and using the Rademacher type condition for centered random variables~\cite[Proposition~9.11]{LT91}, we deduce that
\begin{equation} \label{eq:pluginxu}
\forall \ t\geq0, \qquad\|\Delta P_{2t}f\|_{L_2(\sigma_n;E)}^2 \leq \frac{1}{e^{2t}-1} \mb{E} \Big\|\sum_{i=1}^n \delta_i(t) \partial_iP_tf(\e)\Big\|_E^2 \leq \frac{4T^2}{e^{2t}-1} \sum_{i=1}^n \|\partial_i P_tf\|_{L_2(\sigma_n;E)}^2.
\end{equation}
Plugging~\eqref{eq:pluginxu} in~\eqref{eq:xu} for $\{T_t\}_{t\geq0}=\{P_t\}_{t\geq0}$ \mbox{and doing a change of variables, we get}
\begin{equation}
\begin{split}
\big\|f-&\mb{E}_{\sigma_n}f\big\|_{L_2(\sigma_n;E)}^2 \leq 4C^2 \int_0^\infty \|t\Delta P_{2t}f\|_{L_2(\sigma_n;E)}^2 \frac{\diff t}{t} 
\\ & \leq 8C^2 T^2 \int_0^\infty \frac{2t}{e^{2t}-1} \sum_{i=1}^n \|\partial_i P_tf\|_{L_2(\sigma_n;E)}^2 \diff t \leq 8C^2T^2 \sum_{i=1}^n\int_0^\infty e^{-t} \|\partial_i P_tf\|_{L_2(\sigma_n;E)}^2 \diff t.
\end{split}
\end{equation}
As before, the conclusion now follows from hypercontractivity~\cite{Bon70} along with Lemmas~\ref{lem:orlicz-withpower} and~\ref{lem:orlicz-upper}. 
\end{proof}

\begin{remark}
A careful inspection of the proof of~\cite[Theorem~3.1]{Xu98} shows that if we denote by $X_2(E)$ the least constant $C$ in Xu's inequality~\eqref{eq:xu}, then $X_2(E)\gtrsim M_2(E)$, where $M_2(E)$ is the martingale type 2 constant of $E$. On the other hand, in~\cite{Xu20} it is shown that
\begin{equation} \label{eq:xu-precise}
X_2(E) \lesssim \sup_{t\geq0} \big\|t\partial_t T_t\big\|_{L_2(\mu;E)\to L_2(\mu;E)} M_2(E)
\end{equation}
and the fact that $\sup_{t\geq0} \big\|t\partial_t T_t\big\|_{L_2(\mu;E)\to L_2(\mu;E)}<\infty$ is proven as a consequence of the uniform convexity of $E^\ast$. Specifically for the case of the heat semigroup $\{P_t\}_{t\geq0}$ on $\ms{C}_n$, a different proof of this statement which only relies on Pisier's $K$-convexity theorem~\cite{Pis82} is presented in~\cite[Lemma~37]{EI19}. In the particular case of $E=\ell_p$, where $p\geq2$, an optimization of the argument of~\cite[Lemma~37]{EI19} using the recent proof of Weissler's conjecture on the domain of contractivity of the complex heat flow by Ivanisvili and Nazarov~\cite{IN19}, reveals that
\begin{equation} \label{eq:optimize-ei19}
\forall \ n\in\N, \qquad\sup_{t\geq0} \big\|t\Delta P_t\big\|_{L_2(\sigma_n;\ell_p)\to L_2(\sigma_n;\ell_p)} \lesssim \sqrt{p}.
\end{equation}
Therefore, since the Rademacher and martingale type 2 constants of $\ell_p$ are both of the order of $\sqrt{p}$, the probabilistic proof of Theorem~\ref{thm:useeldan} presented in Section~\ref{subsec:eldan} shows that for every $n\in\N$, every function $f:\ms{C}_n\to\ell_p$, where $p\geq2$, satisfies
\begin{equation*}
\big\| f-\mb{E}_{\sigma_n}f\big\|_{L_2(\sigma_n;\ell_p)}^2 \lesssim p^2 \sum_{i=1}^n \frac{\|\partial_if\|_{L_2(\sigma_n;\ell_p)}^2}{1+\log\big(\|\partial_if\|_{L_2(\sigma_n;\ell_p)}/\|\partial_i f\|_{L_1(\sigma_n;\ell_p)}\big)},
\end{equation*}
whereas the proof via Xu's inequality~\eqref{eq:xu} implies a weaker $O(p^3)$ bound because of the current best known bounds~\eqref{eq:xu-precise} and~\eqref{eq:optimize-ei19}. We refer to~\cite{Xu21a, Xu21b} for recent updates on the optimal order of the constant $X_2(E)$.
\end{remark}

\section{Vector-valued $L_1-L_p$ inequalities} \label{sec:5}

In this section, we will prove Theorems~\ref{thm:lptalagrand} and~\ref{thm:l1lpscalar}. We start by presenting a joint strengthening of the two for functions from the Gauss space instead of the discrete hypercube.


\subsection{A stronger theorem in Gauss space} For a smooth function $f:\R^n\to E$, where $(E,\|\cdot\|_E)$ is a Banach space, and $p\in[1,\infty)$ we will use the shorthand notation
\begin{equation*}
\big\|\nabla f\big\|_{L_p(\gamma_n;E)} \eqdef \Big( \int_{\R^n} \Big\|\sum_{i=1}^n y_i\partial_if\Big\|_{L_p(\gamma_n;E)}^p\diff\gamma_n(y)\Big)^{1/p}.
\end{equation*}
In~\cite[Corollary~2.4]{Pis86}, Pisier presented an argument of Maurey showing that for every Banach space $(E,\|\cdot\|_E)$, $p\in[1,\infty)$ and $n\in\N$, every smooth function $f:\R^n\to E$ satisfies
\begin{equation} \label{eq:maurey-pisier}
\big\|f-\mb{E}_{\gamma_n}f\big\|_{L_p(\gamma_n;E)} \leq \frac{\pi}{2} \ \big\|\nabla f\big\|_{L_p(\gamma_n;E)}.
\end{equation}
In this section, we will prove the following Talagrand-type strengthening of~\eqref{eq:maurey-pisier}.

\begin{theorem} \label{thm:usemp}
For every $p\in(1,\infty)$, there exists $C_p\in(0,\infty)$ such that the following holds. For every Banach space $(E,\|\cdot\|_E)$ and $n\in\N$, every smooth function $f:\R^n\to E$ satisfies
\begin{equation}
\big\|f-\mb{E}_{\gamma_n}f\big\|_{L_p(\gamma_n;E)} \leq C_p \frac{\|\nabla f\|_{L_p(\gamma_n;E)}}{1+\sqrt{\log\big(\|\nabla f\|_{L_p(\gamma_n;E)}/\|\nabla f\|_{L_1(\gamma_n;E)}\big)}}.
\end{equation}
\end{theorem}

We will denote by $\ms{L}$ the (negative) generator of the Ornstein--Uhlenbeck semigroup $\{U_t\}_{t\geq0}$, whose action on a smooth function $f:\R^n\to E$ is given by
\begin{equation*}
\forall \ x\in\R^n, \qquad\ms{L} f(x) = \Delta f(x) - \sum_{i=1}^n x_i\partial_if(x).
\end{equation*}
We will need the following (classical) Gaussian analogue of Proposition~\ref{prop:extendivv}.

\begin{proposition} \label{prop:gaussian-bound}
Let $(E,\|\cdot\|_E)$ be a Banach space and $p\in[1,\infty)$. Then, for every $n\in\N$, every smooth function $f:\R^n\to E$ satisfies 
\begin{equation} \label{eq:gaussian-bound}
\forall \ t\geq0, \qquad\Big\| \frac{\partial}{\partial t} U_t f\Big\|_{L_p(\gamma_n;E)} \leq \frac{1}{\sqrt{e^{2t}-1}} \big\|\nabla f\big\|_{L_p(\gamma_n;E)}.
\end{equation}
\end{proposition}

\begin{proof}
Here we can follow Maurey's trick~\cite{Pis86}, setting
$$X_t= e^{-t} X + \sqrt{1-e^{-2t}} Y \quad\textrm{and}\quad Y_t = -\sqrt{1-e^{-2t}} X+ e^{-t} Y  =  \sqrt{e^{2t}-1}\cdot   \frac{\partial}{\partial t} X_t,$$
for given independent standard Gaussian vectors $X,Y\in \R^n$. 
Then, we have
$$
\frac{\partial}{\partial t} U_t f (X)= \frac{\partial}{\partial t} \E_Y  f(X_t) =  \frac{1}{\sqrt{e^{2t}-1}}  \E_Y \sum_{i=1}^n  \partial_i f(X_t) Y_t(i),
$$
and we conclude using Jensen's inequality together with the fact that $(X_t,Y_t)$ has the same distribution as $(X,Y)$ for every $t\geq0$.
\end{proof}

\begin{proof} [Proof of Theorem~\ref{thm:usemp}]
Arguing as in~\eqref{eq:usetypeforivv1} and using~\eqref{eq:gaussian-bound} for $U_{t}f$ instead of $f$, we can write
\begin{equation} \label{eq:forusemp}
\begin{split}
\big\|f-\mb{E}_{\gamma_n}&f\big\|_{L_p(\gamma_n;E)} \leq 2 \int_0^\infty \|\ms{L}U_{2t}f\|_{L_p(\gamma_n;E)} \diff t
\stackrel{\eqref{eq:gaussian-bound}}{\leq} 2\int_0^\infty \big\|\nabla U_{t}f\big\|_{L_p(\gamma_n;E)} \frac{\diff t}{\sqrt{e^{2t}-1}}
\\ & = 2\int_0^\infty e^{-t} \big\|U_{t}\nabla f\big\|_{L_p(\gamma_n;E)} \frac{\diff t}{\sqrt{e^{2t}-1}} \lesssim \int_0^\infty e^{-t} \big\|U_{t}\nabla f\big\|_{L_p(\gamma_n;E)} \frac{\diff t}{\sqrt{t}}.
\end{split}
\end{equation}
Now, by Nelson's hypercontractive inequalities~\cite{Nel66, Nel73} and Kahane's inequality~\cite{Kah64} for Gaussian variables, we have
\begin{equation} \label{eq:use-kwapien}
\begin{split}
 \big\|U_{t}&\nabla f\big\|_{L_p(\gamma_n;E)}  = \Big(\int_{\R^n} \Big\|\sum_{i=1}^n y_i U_{t}\partial_i f\Big\|_{L_p(\gamma_n;E)}^p\diff\gamma_n(y)\Big)^{1/p}
 \\ & \leq \Big(\int_{\R^n} \Big\|\sum_{i=1}^n y_i \partial_i f\Big\|_{L_{1+(p-1)e^{-2t}}(\gamma_n;E)}^p\diff\gamma_n(y)\Big)^{1/p} \lesssim_p \big\|\nabla f\big\|_{L_{1+(p-1)e^{-2t}}(\gamma_n;E)}
 \end{split}
\end{equation}
and the conclusion follows from~\eqref{eq:forusemp},~\eqref{eq:use-kwapien} and Lemma~\ref{lem:log-nopower}.
\end{proof}


\subsection{Proof of Theorem~\ref{thm:lptalagrand}} Recall that a Banach space $(E,\|\cdot\|_E)$ has cotype $q\in[2,\infty)$ with constant $C\in(0,\infty)$ if for every $n\in\N$ and $x_1,\ldots,x_n\in E$,
\begin{equation}\label{eq:cotype-q}
\sum_{i=1}^n\|x_i\|_E^q \leq C^q \int_{\ms{C}_n} \Big\|\sum_{i=1}^n \e_ix_i\Big\|_E^q\diff\sigma_n(\e).
\end{equation}
The discrete vector-valued $L_1-L_p$ inequality of Theorem~\ref{thm:lptalagrand} can be proven along the same lines as Theorem~\ref{thm:usemp} using Proposition~\ref{prop:extendivv} instead of Proposition~\ref{prop:gaussian-bound}.

\begin{proof} [Proof of Theorem~\ref{thm:lptalagrand}]
Suppose that $E$ has cotype $q\in[2,\infty)$. It has been observed in the proof of~\cite[Proposition~4.2]{IVV20} that~\cite[Proposition~3.2]{Pis86} implies the estimate
\begin{equation} \label{eq:rademachercomparison}
\forall \ t\geq0, \ \ \ \Big( \mb{E} \Big\|\sum_{i=1}^n \delta_i(t) \partial_if(\e)\Big\|_E^p\Big)^{1/p} \leq \frac{B_p}{(1-e^{-2t})^{\frac{1}{2}-\frac{1}{\max\{p,q\}}}}\Big( \mb{E} \Big\|\sum_{i=1}^n \delta_i \partial_if(\e)\Big\|_E^p\Big)^{1/p}
\end{equation}
for some $B_p=B_p(E)\in(0,\infty)$, where $\delta=(\delta_1,\ldots,\delta_n)$ is a random vector, uniformly distributed on $\ms{C}_n$, which is independent of $\e$. Therefore, combining~\eqref{eq:propextendivv},~\eqref{eq:rademachercomparison} and integrating, we deduce that
\begin{equation} \label{eq:lptalagrand00}
\begin{split}
\big\|f-&\mb{E}_{\sigma_n}f\big\|_{L_p(\sigma_n;E)} = 2\Big\| \int_0^\infty \Delta P_{2t}f\diff t\Big\|_{L_p(\sigma_n;E)} \leq 2\int_0^\infty \|\Delta P_{2t}f\|_{L_p(\sigma_n;E)} \diff t
\\ & \stackrel{\eqref{eq:propextendivv}\wedge\eqref{eq:rademachercomparison}}{\leq} 2B_p \int_0^\infty e^{-t} \Big( \mb{E}\Big\|\sum_{i=1}^n\delta_i \partial_i P_t f(\e)\Big\|_E^p\Big)^{1/p} \frac{\diff t}{(1-e^{-2t})^{1-\frac{1}{\max\{p,q\}}}}
\\ & \lesssim B_p \int_0^{\infty} e^{-t/2} \Big( \mb{E}\Big\|\sum_{i=1}^n\delta_i \partial_i P_t f(\e)\Big\|_E^p\Big)^{1/p} \frac{\diff t}{t^{1-\frac{1}{\max\{p,q\}}}}
\end{split}.
\end{equation}
Arguing as in~\eqref{eq:use-kwapien} by using the hypercontractivity of $\{P_t\}_{t\geq0}$ and Kahane's inequality, we get
\begin{equation} \label{eq:lptalagrand11}
\Big( \mb{E}\Big\|\sum_{i=1}^n\delta_i \partial_i P_t f(\e)\Big\|_E^p\Big)^{1/p} \lesssim_p \Big( \mb{E}\Big\|\sum_{i=1}^n\delta_i \partial_i f(\e)\Big\|_E^{p(t)}\Big)^{1/p(t)},
\end{equation}
where $p(t)=1+(p-1)e^{-2t}$ and~\eqref{eq:thmlptalagrand} follows from~\eqref{eq:lptalagrand00},~\eqref{eq:lptalagrand11} and ~\eqref{eq:log-nopower0}  with $\alpha_p(E)=\frac{1}{\max\{p,q\}}$.
\end{proof}

An inspection of the above proofs shows that one can also get the following Orlicz space strengthenings of Theorems~\ref{thm:lptalagrand} and~\ref{thm:usemp} using Lemma~\ref{lem:orlicz-nopower} instead of Lemma~\ref{lem:log-nopower}.

\begin{theorem} \label{thm:lptalagrandorlicz}
Let $(E,\|\cdot\|_E)$ be a Banach space of cotype $q$ and $p\in[1,\infty)$. Then, there exists $C_p=C_p(E)\in(0,\infty)$ such that for every $\theta\in(0,1)$ and $n\in\N$, every $f:\ms{C}_n\to E$ satisfies
\begin{equation}
\big\|f-\mb{E}_{\sigma_n}f\big\|_{L_p(\sigma_n;E)} \leq \frac{C_p}{\theta^{\frac{p-1}{p}}}\cdot \big\|\nabla f\big\|_{L_p(\log L)^{-\frac{p}{\max\{p,q\}}+\theta}(\sigma_n;E)}
\end{equation}
\end{theorem}

\begin{theorem} \label{thm:lpmporlicz}
For every $p\in[1,\infty)$, there exists $C_p\in(0,\infty)$ such that the following holds. For every Banach space $(E,\|\cdot\|_E)$, $\theta\in(0,1)$ and $n\in\N$, every smooth function $f:\R^n\to E$ satisfies
\begin{equation}
\big\|f-\mb{E}_{\gamma_n}f\big\|_{L_p(\gamma_n;E)} \leq \frac{C_p}{\theta^{\frac{p-1}{p}}}\cdot \big\|\nabla f\big\|_{L_p(\log L)^{-\frac{p}{2}+\theta}(\gamma_n;E)}
\end{equation}
\end{theorem}

\subsection{Proof of Theorem~\ref{thm:l1lpscalar}} Since $E=\C$ has cotype 2, the proof of Theorem~\ref{thm:lptalagrand} implies that in the scalar-valued case,~\eqref{eq:thmlptalagrand} holds with an exponent $\alpha_p(\C)=\tfrac{1}{\max\{p,2\}}$ for every $p\in(1,\infty)$. In order to boost this exponent to $\tfrac{1}{2}$ we shall use the following deep result of Lust-Piquard~\cite{LP98} (see also~\cite{BELP08} for a slightly neater argument with better dependence on $p$ and further applications).

\begin{theorem} [Lust-Piquard] \label{thm:lust-piquard}
For every $p\in(1,\infty)$, there exists $\beta_p\in(0,\infty)$ such that for every $n\in\N$, every function $f:\ms{C}_n\to\C$ satisfies
\begin{equation} \label{eq:lust-piquard}
\beta_p \big\|\Delta^{1/2}f\big\|_{L_p(\sigma_n)} \leq \big\|\nabla f\big\|_{L_p(\sigma_n)}.
\end{equation}
\end{theorem}

\begin{proof} [Proof of Theorem~\ref{thm:l1lpscalar}]
By Khintchine's inequality~\cite{Khi23}, for every function $f:\ms{C}_n\to\C$, we have
\begin{equation} \label{eq:khi00}
\forall \ \e\in\ms{C}_n, \qquad\Big(\sum_{i=1}^n \big(\partial_i f(\e)\big)^2\Big)^{1/2} \asymp_p \Big(\mb{E}\Big|\sum_{i=1}^n \delta_i \partial_if(\e)\Big|^p\Big)^{1/p},
\end{equation}
where the expectation is with respect to $\delta=(\delta_1,\ldots,\delta_n)$ uniformly distributed on $\ms{C}_n$. Therefore, if $F:\ms{C}_n\to L_p(\sigma_n)$ is given by $\big[F(\e)\big](\delta) = \sum_{i=1}^n \delta_i \partial_if(\e)$,~\eqref{eq:lust-piquard},~\eqref{eq:khi00} and Theorem~\ref{thm:naorschechtman} imply that
\begin{equation*}
\begin{split}
\big\|f-\mb{E}_{\sigma_n}f\big\|_{L_p(\sigma_n)} & \stackrel{\eqref{eq:lust-piquard}}{\lesssim_p} \big\| \nabla \Delta^{-1/2} f\big\|_{L_p(\sigma_n)} \stackrel{\eqref{eq:khi00}}{\asymp_p} \big\|\Delta^{-1/2} F\big\|_{L_p(\sigma_n;L_p(\sigma_n))}
\\ & \stackrel{\eqref{eq:thmnaorschechtman}}{\lesssim_p} \frac{\|F\|_{L_p(\sigma_n;L_p(\sigma_n))}}{1+\sqrt{\log\big(\|F\|_{L_p(\sigma_n;L_p(\sigma_n))}/ \|F\|_{L_1(\sigma_n;L_p(\sigma_n))}\big)}}.
\end{split}
\end{equation*}
The conclusion now follows since, again by Khintchine's inequality~\eqref{eq:khi00}, the function $F$ satisfies $\|F\|_{L_p(\sigma_n;L_p(\sigma_n))}\asymp_p \|\nabla f\|_{L_p(\sigma_n)}$ and $\|F\|_{L_1(\sigma_n;L_p(\sigma_n))}\asymp_p \|\nabla f\|_{L_1(\sigma_n)}$.
\end{proof}

\begin{remark} \label{rem:l1l2scalar}
We note in passing that for $p=2$,~\eqref{eq:thml1lpscalar} is a consequence of Talagrand's influence inequality~\eqref{eq:talagrand}. To see this, note that it has been observed in~\cite[Theorem~5.4]{Cha14} that Talagrand's inequality~\eqref{eq:talagrand} along with an application of Jensen's inequality imply that for every $n\in\N$, every $f:\ms{C}_n\to\C$ satisfies
\begin{equation*}
\mathrm{Var}_{\sigma_n}(f) \leq C\frac{\|\nabla f\|_{L_2(\sigma_n)}^2}{1+\log(u(f))},
\end{equation*}
where $u(f) = (\sum_{i=1}^n \|\partial_if\|_{L_2(\sigma_n)}^2)/(\sum_{i=1}^n \|\partial_if\|_{L_1(\sigma_n)}^2)$ and $C\in(0,\infty)$ is a universal constant. Then,~\eqref{eq:thml1lpscalar} for $p=2$ follows by Minkowski's integral inequality, since
\begin{equation*}
u(f) = \frac{\sum_{i=1}^n \|\partial_if\|_{L_2(\sigma_n)}^2}{\sum_{i=1}^n \|\partial_if\|_{L_1(\sigma_n)}^2} \geq \frac{\|\nabla f\|_{L_2(\sigma_n)}^2}{\|\nabla f\|_{L_1(\sigma_n)}^2}.
\end{equation*}
\end{remark}

Using the vector-valued Bakry--Meyer inequality of Theorem~\ref{thm:vectorbakrymeyer} instead of Theorem~\ref{thm:naorschechtman}, one obtains the following Orlicz space strengthening of Theorem~\ref{thm:l1lpscalar}.

\begin{theorem} \label{thm:l1lporlicz}
For every $p\in(1,\infty)$, there exists $C_p\in(0,\infty)$ such that for every $n\in\N$, every $f:\ms{C}_n\to\C$ satisfies
\begin{equation} \label{eq:thml1lporlicz}
\big\| f-\mb{E}_{\sigma_n}f\big\|_{L_p(\sigma_n)} \leq C_p \big\|\nabla f\big\|_{L_p(\log L)^{-p/2}(\sigma_n)}.
\end{equation}
\end{theorem}

\section{Holomorphic multipliers and the vector-valued Bakry--Meyer theorem} \label{sec:6}

In this section, we will present the proofs of Theorems~\ref{thm:naorschechtman},~\ref{thm:vectorbakrymeyer} and~\ref{thm:multiplier}. In the proof of Theorem~\ref{thm:multiplier}, we will need some preliminary terminology from discrete Fourier analysis. Recall that for every Banach space $(E,\|\cdot\|_E)$ and every $n\in\N$, all functions $f:\ms{C}_n\to E$ admit a unique expansion of the form
\begin{equation*}
\forall \ \e\in\ms{C}_n, \qquad f(\e) = \sum_{A\subseteq\{1,\ldots,n\}} \widehat{f}(A) w_A(\e),
\end{equation*}
where the Walsh function $w_A:\ms{C}_n\to\{-1,1\}$ is given by $w_A(\e) = \prod_{i\in A}\e_i$ for $\e\in\ms{C}_n$. In this basis, the action of the hypercube Laplacian on $f$ can be written as
\begin{equation*}
\Delta f = \sum_{A\subseteq\{1,\ldots,n\}} |A|\widehat{f}(A) w_A.
\end{equation*}
Suppose now that $r\in(0,\infty)$ and that $h:(0,r)\to\C$ is a function. Then, for every $\alpha\in(0,\infty)$, the operator $h(\Delta^{-\alpha})$ is defined spectrally by
\begin{equation}\label{eq:defoperator}
h\big(\Delta^{-\alpha}\big) \eqdef \sum_{\substack{A\subseteq\{1,\ldots,n\} \\ |A|> r^{-1/\alpha}}} h\big(|A|^{-\alpha}\big) \widehat{f}(A) w_A.
\end{equation}
Finally, for a function $f:\ms{C}_n\to E$ and $k\in\{0,1,\ldots,n\}$ we will define the $k$-th level Rademacher projection of $f$ to be the function with Walsh expansion
\begin{equation*}
\msf{Rad}_k f \eqdef \sum_{\substack{A\subseteq\{1,\ldots,n\} \\ |A|=k}} \widehat{f}(A) w_A.
\end{equation*}
Pisier's deep $K$-convexity theorem~\cite{Pis82} asserts that a Banach space $(E,\|\cdot\|_E)$ has nontrivial Rademacher type if and only if for every $p\in(1,\infty)$, there exist $M_p=M_p(E)\in(0,\infty)$ such that for every $n\in\N$ and $k\in\{1,\ldots,n\}$, every $f:\ms{C}_n\to E$ satisfies $\|\msf{Rad}_kf\|_{L_p(\sigma_n;E)} \leq M_p^k \|f\|_{L_p(\sigma_n;E)}$.


\subsection{Proof of Theorem~\ref{thm:naorschechtman}} Although Theorem~\ref{thm:naorschechtman} is a formal consequence of Theorem~\ref{thm:vectorbakrymeyer} and Lemma~\ref{lem:orlicz-upper}, we present a short self-contained proof.

\begin{proof}[Proof of Theorem~\ref{thm:naorschechtman}]
Since $P_t = e^{-t\Delta}$, we can express the action of $\Delta^{-\alpha}$ on functions with expectation equal to 0 as
\begin{equation} \label{eq:Delta-a}
\Delta^{-\alpha} = \frac{1}{\Gamma(\alpha)} \int_0^\infty P_t\ \frac{\diff t}{t^{1-\alpha}}.
\end{equation}
Hence, every function $f:\ms{C}_n\to E$ with $\mb{E}_{\sigma_n}f=0$ satisfies
\begin{equation} \label{eq:nstrick}
\big\|\Delta^{-\alpha}f\big\|_{L_p(\sigma_n;E)} \leq \frac{1}{\Gamma(\alpha)} \int_0^\infty \|P_tf\|_{L_p(\sigma_n;E)} \frac{\diff t}{t^{1-\alpha}}.
\end{equation}
If $E$ has nontrivial type, it is a standard consequence of Pisier's $K$-convexity theorem~\cite{Pis82} that there exist $K_p=K_p(E)\in(0,\infty)$ and $\eta_p=\eta_p\in\big(0,\tfrac{1}{2}\big]$, independent of $n$ and $f$, such that
\begin{equation} \label{eq:decayKconv}
\mb{E}_{\sigma_n}f=0 \quad\Longrightarrow \quad \forall \ t\geq0, \quad\|P_{t}f\|_{L_p(\sigma_n;E)} \leq K_pe^{-2\eta_pt}\|f\|_{L_p(\sigma_n;E)}.
\end{equation}
Combining~\eqref{eq:nstrick} and~\eqref{eq:decayKconv}, we deduce that
\begin{equation*}
\big\|\Delta^{-\alpha}f\big\|_{L_p(\sigma_n;E)} \leq \frac{K_p}{\Gamma(\alpha)} \int_0^\infty e^{-\eta_p t} \|P_{t/2}f\|_{L_p(\sigma_n;E)} \frac{\diff t}{t^{1-\alpha}}
\end{equation*}
and the conclusion follows by hypercontractivity~\cite{Bon70} and Lemma~\ref{lem:log-nopower}.
\end{proof}

\subsection{Proof of Theorem~\ref{thm:multiplier}} The proof of Theorem~\ref{thm:multiplier} relies on the following result of Mendel and Naor from~\cite{MN14} (see also~\cite{EI19} for a different proof and further results in this direction).

\begin{theorem} [Mendel--Naor]
Let $(E,\|\cdot\|_E)$ be a Banach space of nontrivial type and $p\in(1,\infty)$. Then, there exist $c_p=c_p(E), C_p=C_p(E)\in(0,\infty)$ and $A_p=A_p(E)\in[1,\infty)$ such that for every $n\in\N$ and $d\in\{1,\ldots,n\}$, the following holds. Every function $f:\ms{C}_n\to E$ whose Fourier coefficients $\widehat{f}(A)$ vanish for all subsets $A\subseteq\{1,\ldots,n\}$ with $|A|<d$ satisfies
\begin{equation} \label{eq:mendel-naor}
\|P_tf\|_{L_p(\sigma_n;E)} \leq C_p e^{-c_pd\min\{t,t^{A_p}\}} \|f\|_{L_p(\sigma_n;E)}.
\end{equation}
\end{theorem}
Using identity~\eqref{eq:Delta-a} and~\eqref{eq:mendel-naor}, we see that every such function $f:\ms{C}_n\to E$ satisfies
\begin{equation} \label{eq:laplace-tail-space}
\frac{\|\Delta^{-\alpha}f\|_{L_p(\sigma_n;E)}}{\|f\|_{L_p(\sigma_n;E)}} \leq \frac{C_p}{\Gamma(\alpha)} \int_0^1 e^{-c_p d t^{A_p}} \frac{\diff t}{t^{1-\alpha}} +  \frac{C_p}{\Gamma(\alpha)} \int_1^\infty e^{-c_p d t} \frac{\diff t}{t^{1-\alpha}} \leq \frac{K_p(\alpha)}{d^{\alpha/A_p}},
\end{equation}
for some $K_p(\alpha)=K_p(\alpha,E)\in(0,\infty)$.

\begin{proof} [Proof of Theorem~\ref{thm:multiplier}]
Let $d_p(\alpha) = \left\lceil (2K_p(\alpha)/r)^{A_p/\alpha}\right\rceil$, where $K_p(\alpha)$ is the same as in~\eqref{eq:laplace-tail-space}, so that every function $f:\ms{C}_n\to E$ whose Fourier coefficients $\widehat{f}(A)$ vanish for all subsets $A\subseteq\{1,\ldots,n\}$ with $|A|<d_p(\alpha)$ satisfies $\|\Delta^{-\alpha}f\|_{L_p(\sigma_n;E)} \leq \tfrac{r}{2} \|f\|_{L_p(\sigma_n;E)}$. Iterating this inequality, we get
\begin{equation} \label{eq:tensorize-mendel-naor}
\forall \ \ell\geq1, \qquad\|\Delta^{-\alpha \ell}f\|_{L_p(\sigma_n;E)} \leq \Big(\frac{r}{2}\Big)^\ell \|f\|_{L_p(\sigma_n;E)}
\end{equation}
for every such function $f$.

Now, let $f:\ms{C}_n\to E$ be an arbitrary function and write
\begin{equation*}
\forall \ \e\in\ms{C}_n, \qquad  f(\e) = \underbrace{\sum_{k=0}^{d_p(\alpha)-1} \msf{Rad}_kf(\e)}_{f_1(\e)} + \underbrace{\sum_{k=d_p(\alpha)}^{n} \msf{Rad}_kf(\e)}_{f_2(\e)}.
\end{equation*}
By Pisier's $K$-convexity theorem~\cite{Pis82}, we have
\begin{equation} \label{eq:meyer11}
\begin{split}
\big\|h\big(\Delta^{-\alpha}\big)f_1\big\|_{L_p(\sigma_n;E)} \leq  \sum_{k=\lfloor r^{-1/\alpha}\rfloor+1}^{d_p(\alpha)-1} &\big|h(k^{-\alpha})\big| \|\msf{Rad}_kf\|_{L_p(\sigma_n;E)} 
\\ & \leq \bigg( \sum_{k=\lfloor r^{-1/\alpha}\rfloor+1}^{d_p(\alpha)-1}\big|h(k^{-\alpha})\big| M_p^k\bigg) \|f\|_{L_p(\sigma_n;E)}
\end{split},
\end{equation}
for some $M_p=M_p(E)\in(0,\infty)$. To bound the action of $h(\Delta^{-\alpha})$ on $f_2$, consider the power series expansion $h(z)=\sum_{\ell\geq0} c_\ell z^\ell$ of $h$ around 0, which converges absolutely and uniformly on $\overline{\mb{D}}_{r/2}$. Then, the triangle inequality implies that
\begin{equation} \label{eq:meyer22}
\big\|h\big(\Delta^{-\alpha}\big)f_2\big\|_{L_p(\sigma_n;E)} \leq \sum_{\ell\geq0} |c_\ell| \big\| \Delta^{-\alpha\ell} f_2\big\|_{L_p(\sigma_n;E)} \stackrel{\eqref{eq:tensorize-mendel-naor}}{\leq}  \bigg( \sum_{\ell\geq0} |c_\ell|\Big(\frac{r}{2}\Big)^\ell \bigg) \|f_2\|_{L_p(\sigma_n;E)}.
\end{equation}
Finally, observe that, again by Pisier's $K$-convexity theorem,
\begin{equation} \label{eq:meyer33}
\begin{split}
\|f_2\|_{L_p(\sigma_;E)} = \|f-f_1\|_{L_p(\sigma_n;E)} \leq \|f\|_{L_p(\sigma_n;E)} & + \sum_{k=0}^{d_p(\alpha)-1} \|\msf{Rad}_kf\|_{L_p(\sigma_n;E)} 
\\ & \leq \bigg( 1+\sum_{k=0}^{d_p(\alpha)-1} M_p^k\bigg) \|f\|_{L_p(\sigma_n;E)},
\end{split}
\end{equation}
for some $M_p=M_p(E)\in(0,\infty)$. The conclusion follows readily from~\eqref{eq:meyer11},~\eqref{eq:meyer22} and~\eqref{eq:meyer33}. 
\end{proof}

\subsection{Proof of Theorem~\ref{thm:vectorbakrymeyer}} Equipped with Theorem~\ref{thm:multiplier}, we can now deduce Theorem~\ref{thm:vectorbakrymeyer} from~\eqref{eq:bakrymeyer}. We will also need the following simple lemma.

\begin{lemma} \label{lem:Delta+1}
For every Banach space $(E,\|\cdot\|_E)$, every function $f:\ms{C}_n\to E$ and every $\alpha\in(0,\infty)$,
\begin{equation} \label{eq:lemDelta+1}
\forall \ \e\in\ms{C}_n, \qquad\big\| (\Delta+1)^{-\alpha}f(\e)\big\|_E \leq \big[(\Delta+1)^{-\alpha}\|f\|_E\big](\e).
\end{equation}
\end{lemma}

\begin{proof}
A change of variables shows that
\begin{equation} \label{eq:Delta+1}
\big(\Delta+1\big)^{-\alpha} = \frac{1}{\Gamma(\alpha)} \int_0^\infty e^{-t} P_t \ \frac{\diff t}{t^{1-\alpha}},
\end{equation}
so that for every $\e\in\ms{C}_n$, we have
\begin{equation*}
\begin{split}
\big\| (\Delta+1)^{-\alpha}f(\e)\big\|_E & \leq \frac{1}{\Gamma(\alpha)} \int_0^\infty e^{-t} \big\|P_tf(\e)\big\|_E \frac{\diff t}{t^{1-\alpha}}
 \\ & \leq \frac{1}{\Gamma(\alpha)}\int_0^\infty e^{-t} \big[P_t\|f\|_E\big](\e) \frac{\diff t}{t^{1-\alpha}} = \big[(\Delta+1)^{-\alpha}\|f\|_E\big](\e),
\end{split}
\end{equation*}
where the second inequality follows from Jensen's inequality because $P_t$ is an averaging operator.
\end{proof}

\begin{proof} [Proof of Theorems~\ref{thm:naorschechtman} and~\ref{thm:vectorbakrymeyer}]
Let $\phi,\psi:\mb{D}_1\to \C$ be two holomorphic branches of
\begin{equation*}
\forall \ z\in\mb{D}_1, \qquad\phi(z) = (1+z)^\alpha \ \ \ \mbox{and} \ \ \ \psi(z)=(1+z)^{-\alpha}
\end{equation*}
on $\mb{D}_1$. Then, by Theorem~\ref{thm:multiplier}, the operators $\phi(\Delta^{-1})$ and $\psi(\Delta^{-1})$ are bounded on $L_p(\sigma_n;E)$, where $p\in(1,\infty)$, with operator norms independent of $n$. In other words, there exist constants $\lambda_p(\alpha,E), \Lambda_p(\alpha,E)\in(0,\infty)$ such that for every $n\in\N$, every function $f:\ms{C}_n\to E$ satisfies
\begin{equation} \label{eq:lambdas}
\lambda_p(\alpha,E) \|\Delta^{-\alpha}f\|_{L_p(\sigma_n;E)} \leq \big\|(\Delta+1)^{-\alpha}f\big\|_{L_p(\sigma_n;E)} \leq \Lambda_p(\alpha,E) \|\Delta^{-\alpha}f\|_{L_p(\sigma_n;E)}.
\end{equation}
Combining~\eqref{eq:lambdas} with Lemma~\ref{lem:Delta+1} and the\mbox{ inequality~\eqref{eq:bakrymeyer} of Bakry and Meyer~\cite{BM82}, we get}
\begin{equation*}
\begin{split}
\|\Delta^{-\alpha}f\|_{L_p(\sigma_n;E)}& \stackrel{\eqref{eq:lambdas}}{\leq} \lambda_p(\alpha,E)^{-1} \big\|(\Delta+1)^{-\alpha}f\big\|_{L_p(\sigma_n;E)} \stackrel{\eqref{eq:lemDelta+1}}{\leq} \lambda_p(\alpha,E)^{-1} \big\|(\Delta+1)^{-\alpha}\|f\|_E\big\|_{L_p(\sigma_n)}
\\ & \stackrel{\eqref{eq:lambdas}}{\leq} \frac{\Lambda_p(\alpha,\C)}{\lambda_p(\alpha,E)}  \big\|\Delta^{-\alpha}\|f\|_E\big\|_{L_p(\sigma_n)} \stackrel{\eqref{eq:bakrymeyer}}{\leq}  \frac{\Lambda_p(\alpha,\C)K_p(\alpha)}{\lambda_p(\alpha,E)} \|f\|_{L_p(\log L)^{-p\alpha}(\sigma_n;E)},
\end{split}
\end{equation*}
for some $K_p(\alpha)\in(0,\infty)$ and the conclusion of Theorem~\ref{thm:vectorbakrymeyer} follows.
\end{proof}

\section{Influence inequalities in nonpositive curvature} \label{sec:gromov}

Theorems~\ref{thm:gromov} and~\ref{thm:pinched} will be proven  by combining Theorem~\ref{thm:useivvorlicz} with results from geometry and Banach space theory. We first prove Theorem~\ref{thm:gromov}.

\begin{proof} [Proof of Theorem~\ref{thm:gromov}]
It immediately follows from definition~\eqref{eq:deftalagrand} that if a metric space $\MM$ has Talagrand type $(p,\psi)$ with constant $\tau\in(0,\infty)$ and another metric space $\NN$ embeds bi-Lipscitzly in $\MM$ with distortion $D\in[1,\infty)$, then $\NN$ has Talagrand type $(p,\psi)$ with constant $\tau D$. Let $\msf{G}$ be a Gromov hyperbolic group equipped with the shortest path metric $d_\msf{G}$ associated to the Cayley graph of any (finite) generating set $\msf{S}$. Then, by a theorem of Ostrovskii~\cite{Ost14}, $(\msf{G},d_\msf{G})$ admits a bi-Lipschitz embedding of bounded distortion into any nonsuperreflexive Banach space. In particular, $(\msf{G},d_\msf{G})$ embeds bi-Lipschitzly in the classical exotic Banach space $(\mb{J},\|\cdot\|_\mb{J})$ of James~\cite{Jam78}, which has Rademacher type 2 yet is not superreflexive. By Theorem~\ref{thm:useivvorlicz}, there exists a universal constant $C\in(0,\infty)$ such that for every $\e\in(0,1)$, $(\mb{J},\|\cdot\|_\mb{J})$ has Talagrand type $(2,\psi_{2,1-\e})$ with constant $C/\sqrt{\e}$ and thus the same holds true for the group $(\msf{G},d_\msf{G})$.
\end{proof}

The binary $\mathbb{R}$-tree of depth $d$ is the geodesic metric space which is obtained by replacing every edge of the combinatorial binary tree of depth $d$ by the interval $[0,1]$. In order to prove Theorem~\ref{thm:pinched}, we will need the following structural result for Riemannian manifolds of pinched negative curvature which is essentially due to Naor, Peres, Schramm and Sheffield~\cite{NPSS06}.

\begin{theorem} \label{thm:npss}
Fix $n\in\N$ and $r,R\in(0,\infty)$ with $r<R$. Then, there exists $N\in\N$ and $D\in(0,\infty)$ such that any $n$-dimensional complete simply connected Riemannian manifold $(\msf{M},g)$ with sectional curvature in $[-R,-r]$ embeds bi-Lipschitzly with distortion at most $D$ in a product \mbox{of $N$ binary $\R$-trees of infinite depth.}
\end{theorem}

In~\cite[Corollary~6.5]{NPSS06}, the authors proved an analogue of Theorem~\ref{thm:npss}, in which binary $\R$-trees are replaced by $\R$-trees of infinite degree. In order to prove the (stronger) theorem presented here, one needs to repeat the argument of \cite{NPSS06} verbatim, replacing the use of \cite{BS05} with a more recent result of Dranishnikov and Schroeder \cite{DS05}, who showed that the hyperbolic space $\mb{H}^m$ admits a quasi-isometric in a finite product of binary $\R$-trees of infinite depth.

We shall also need the following slight refinement of a result of Bourgain~\cite{Bou86}.

\begin{proposition} \label{prop:bourgain}
Let $(E,\|\cdot\|_E)$ be a nonsuperreflexive Banach space. For every $d\in\N$, the binary $\R$-tree of depth $d$ embeds in $E$ with distortion at most 4.
\end{proposition}

\begin{proof}
Fix $d\in\N$, let $\mb{B}_d$ be the combinatorial binary tree of depth $d$ and denote its root by $r$. There exists a natural enumeration $\sigma:\mb{B}_d\to\{1,\ldots,2^{d+1}-1\}$ of the vertices of $\mb{B}_d$ with the following property: if $x,y$ are two leaves of the tree whose least common ancestor is $z$, then $\sigma\big((z,x]\big)$ and $\sigma\big((z,y]\big)$ are two disjoint subsets of $\{1,\ldots,2^{d+1}-1\}$ such that one of the inequalities
\begin{equation} \label{eq:labeling-for-james}
\max\sigma\big((z,x]\big) < \min\sigma\big((z,y]\big) \ \ \  \mbox{or} \ \ \ \max\sigma\big((z,y]\big) < \min\sigma\big((z,x]\big)
\end{equation}
holds true. To see this, one can ``draw'' the binary tree and label the vertices from top to bottom along an arbitrary path. After reaching a leaf, one should return to the nearest ancestor with an unlabeled child and continue labeling along an arbitrary downwards path starting at this child. This process should continue until the whole tree has been labeled. 

Since $E$ is nonsuperreflexive, by a classical theorem of Pt\'ak~\cite[Theorem~11.10]{Pis16} (which is often attributed to James), there exists vectors $\{x_k\}_{k=1}^{2^{d+1}-1}$ such that for every scalars $a_1,\ldots,a_{2^{d+1}-1}$,
\begin{equation} \label{eq:james-condition}
\frac{1}{4}\sup_{j\in\{1,\ldots,2^{d+1}-1\}} \left\{ \Big| \sum_{i<j} a_i\Big| + \Big| \sum_{i\geq j} a_i\Big|\right\} \leq \Big\| \sum_{i=1}^{2^{d+1}-1} a_i x_i\Big\|_E \leq \sum_{i=1}^{2^{d+1}-1} |a_i|.
\end{equation}
Let $\overline{\mb{B}}_d$ be the binary $\mb{R}$-tree of depth $d$. For a point $a\in\overline{\mb{B}}_d$ suppose that $a$ belongs in the edge $\{v,w\}$ of $\mb{B}_d$ and that $v$ is closer to the root than $w$. Consider the embedding $\psi:\overline{\mb{B}}_d\to E$ given by
\begin{equation*}
\psi(a) \eqdef \sum_{u\in[r,a]\cap\mb{B}_d} x_{\sigma(u)} + d_{\overline{\mb{B}}_d}(v,a) \cdot x_{\sigma(w)}.
\end{equation*}
Let $a,b\in\overline{\mb{B}}_d$ and suppose that $c$ is their least common ancestor. Then, there are downwards paths $\{s_1,\ldots,s_{j+1}\}$, $\{t_1,\ldots,t_{k+1}\}$ in $\mb{B}_d$ such that $a\in[s_j,s_{j+1})$, $b\in[t_k,t_{k+1})$ and $s_1, t_1$ are the two distinct children of $c$. In this notation, the embedding $\psi$ satisfies
\begin{equation} \label{eq:embedding-difference}
\psi(a) -\psi(b) = \sum_{i=1}^j x_{\sigma(s_i)} + \delta x_{\sigma(s_{j+1})} - \sum_{i=1}^k x_{\sigma(t_i)} - \e x_{\sigma_(t_{k+1})},
\end{equation}
where $\delta = d_{\overline{\mb{B}}_d}(s_j,a)$ and $\e= d_{\overline{\mb{B}}_d}(t_k,b)$. Since $\|x_i\|_X\leq 1$, it is clear that
\begin{equation*}
\big\|\psi(a)-\psi(b)\big\|_E \leq j+\delta+k+\e =  d_{\overline{\mb{B}}_d}(a,b).
\end{equation*}
On the other hand, by the property~\eqref{eq:labeling-for-james} of $\sigma$, we can assume without loss of generality that
\begin{equation*}
\max\big\{\sigma(s_1),\ldots,\sigma(s_{j+1})\big\} < \min\big\{\sigma(t_1),\ldots,\sigma(t_{k+1})\big\}.
\end{equation*}
Then,~\eqref{eq:embedding-difference} and~\eqref{eq:james-condition} imply that
\begin{equation*}
\big\|\psi(a)-\psi(b)\big\|_E  \geq \frac{1}{4}\big( j +\delta+k+\e\big) = \frac{d_{\overline{\mb{B}}_d}(a,b)}{4}.
\end{equation*}
Therefore, $\psi$ is the desired bi-Lipschitz embedding.
\end{proof}

\begin{proof} [Proof of Theorem~\ref{thm:pinched}]
It follows from definition~\eqref{eq:deftalagrand} that if a metric space $\MM$ has Talagrand type $(p,\psi)$ with constant $\tau\in(0,\infty)$ and another metric space $\NN$ is such that every finite subset of $\NN$ embeds bi-Lipscitzly in $\MM$ with distortion at most $K\in[1,\infty)$, then $\NN$ has Talagrand type $(p,\psi)$ with constant $\tau K$. Let $(\msf{M},g)$ be a Riemannian manifold of pinched negative curvature equipped with its Riemannian distance $d_\msf{M}$. Then, by Theorem~\ref{thm:npss}, there exists $N\in\N$ and $D\in(0,\infty)$ such that $(\msf{M},d_\msf{M})$ embeds with distortion at most $D$ in a product of $N$ binary $\R$-trees of infinite depth. In particular, every finite subset $\msf{X}$ of $\msf{M}$ embeds with distortion at most $D$ in a product of $N$ binary $\R$-trees of depth $d$, for some $d$ depending on the cardinality of $\msf{X}$. Therefore, by Proposition~\ref{prop:bourgain} (see also the discussion following Theorem~2.1 in~\cite{Ost14}), $\msf{X}$ embeds with distortion at most $K=K(N,D)\in(0,\infty)$ in every nonsuperreflexive Banach space. In particular, $\msf{X}$ embeds with distortion at most $K$ in the classical exotic Banach space $(\mb{J},\|\cdot\|_{\mb{J}})$ of James~\cite{Jam78}, which has Rademacher type 2 yet is not superreflexive. By Theorem~\ref{thm:useivvorlicz}, there exists a universal constant $C\in(0,\infty)$ such that for every $\e\in(0,1)$, $(\mb{J},\|\cdot\|_{\mb{J}})$ has Talagrand type $(2,\psi_{2,1-\e})$ with constant $C/\sqrt{\e}$ and thus the same holds for the Riemannian manifold $(\msf{M},d_\msf{M})$.
\end{proof}

\section{Embeddings of nonlinear quotients of the cube and Talagrand type} \label{sec:embed}

We will now now prove that Talagrand type is an obstruction to embeddings of quotients of $\ms{C}_n$.

\begin{proof} [Proof of Theorem \ref{thm:embed}]
Suppose that $(\MM,d_\MM)$ has Talagrand type $(p,\psi)$ with constant $\tau$ and let $\ms{R}\subseteq\ms{C}_n\times\ms{C}_n$ be an equivalence relation. Let $f:\ms{C}_n/\ms{R}\to\MM$ be a map satisfying
\begin{equation} \label{bilip}
\forall \ [\zeta], [\eta]\in\ms{C}_n/\ms{R}, \qquad s \rho_{\ms{C}_n/\ms{R}}\big([\zeta],[\eta]\big) \leq d_\MM\big( f([\zeta]), f([\eta])\big) \leq s D \rho_{\ms{C}_n/\ms{R}}\big([\zeta],[\eta]\big),
\end{equation}
where $s\in(0,\infty)$ and $D\geq1$. Consider the lifting $F:\ms{C}_n\to\MM$ given by $F(\e) = f([\e])$, where $\e\in\ms{C}_n$. Then, since $\MM$ has Talagrand type $(p,\psi)$ with constant $\tau$, we have
\begin{equation} \label{eq:talrepeat}
\int_{\ms{C}_n\times\ms{C}_n} d_\MM\big(F(\e),F(\delta)\big)^p\diff\sigma_{2n}(\e,\delta) \leq \tau^p \sum_{i=1}^n \|\mathfrak{d}_iF\|^p_{L_\psi(\sigma_n)}.
\end{equation}
The bi-Lipschitz condition \eqref{bilip} and the definition of $F$ imply that
\begin{equation} \label{lowerb}
\forall \ \e,\delta\in\ms{C}_n, \qquad d_\MM\big(F(\e),F(\delta)\big) = d_\MM\big(f\big([\e]\big), f\big([\delta]\big)\big) \stackrel{\eqref{bilip}}{\geq} s \rho_{\ms{C}_n/\ms{R}}\big([\e],[\delta]\big).
\end{equation}
On the other hand, for every $\e\in\ms{C}_n$,
\begin{equation*}
\begin{split}
\mathfrak{d}_iF(\e) =  \frac{1}{2} d_\MM&\big(F(\e),F(\e_1,\ldots,\e_{i-1},-\e_i,\e_{i+1},\ldots,\e_n)\big) \\ & \stackrel{\eqref{bilip}}{\leq} \frac{sD}{2} \rho_{\ms{C}_n/\ms{R}} \big([\e],[(\e_1,\ldots,\e_{i-1},-\e_i,\e_{i+1},\ldots,\e_n)]\big) = \frac{sD}{2} {\bf 1}_{\partial_i \ms{R}}(\e),
\end{split}
\end{equation*}
since $\rho_{\ms{C}_n/\ms{R}} ([\e],[(\e_1,\ldots,\e_{i-1},-\e_i,\e_{i+1},\ldots,\e_n)])\in\{0,1\}$ for every $\e\in\ms{C}_n$ and it vanishes if and only if $ (\e,(\e_1,\ldots,\e_{i-1},-\e_i,\e_{i+1},\ldots,\e_n))\in\ms{R}$. Therefore,
\begin{equation} \label{upperb}
\|\mathfrak{d}_iF\|_{L_\psi(\sigma_n)} \leq \frac{sD}{2} \| {\bf 1}_{\partial_i \ms{R}}\|_{L_\psi(\sigma_n)} = \frac{sD}{2\psi^{-1}\big(\sigma_n(\partial_i\ms{R})^{-1}\big)}.
\end{equation}
Combining \eqref{eq:talrepeat}, \eqref{lowerb} and \eqref{upperb}, we deduce that
\begin{equation*}
\frac{s^pD^p\tau^p}{2^p} \sum_{i=1}^n \psi^{-1}\big(\sigma_n(\partial_i\ms{R})^{-1}\big)^{-p} \geq s^p \msf{a}_p(\ms{R})^p
\end{equation*}
and the conclusion follows.
\end{proof}

\begin{remark} \label{rem:enflo-def}
A metric space $(\MM, d_\MM)$ is said to have Enflo type $p\in(0,\infty)$ with constant $T\in(0,\infty)$ if for every $n\in\N$, every function $f:\ms{C}_n\to\MM$ satisfies
\begin{equation}
\int_{\ms{C}_n} d_\MM\big(f(\e),f(-\e)\big)^p \diff\sigma_n(\e) \leq T^p \sum_{i=1}^n  \|\mathfrak{d}_iF\|^p_{L_p(\sigma_n)}.
\end{equation}
While Talagrand type is meant to be a refinement of Enflo type (where the Young function is $\psi(t)=t^p$), the attentive reader will notice that the left-hand-sides of the two inequalities are different. This difference is mainly superficial (and originates from Enflo's original definition of ``roundedness'' of a metric space, see \cite{Enf69}) and all interesting geometric applications of Enflo type could be recovered with either definition. Since we discuss the bi-Lipschitz geometry of quotients of $(\ms{C}_n,\rho)$, it is more natural to define Talagrand type by \eqref{eq:deftalagrand} in order to be able to get distortion lower bounds for quotients $\ms{C}_n/\ms{R}$ satisfying $(\e,-\e)\in\ms{R}$ for every $\e\in\ms{C}_n$.
\end{remark}

\begin{remark} \label{rem:khot-naor}
Theorem \ref{thm:embed} provides distortion lower bounds for the embedding of quotients of $(\ms{C}_n,\rho)$ by an arbitrary equivalence relation $\ms{R}$ into spaces with prescribed Talagrand type. While we are not aware of any such bounds in the literature (except perhaps the bound \eqref{eq:weak-embed} which one can deduce from Enflo type $p$), it is worth mentioning that there exist $L_p$-nonembeddability results for more structured quotients of $\ms{C}_n$. In particular, we refer the reader to the paper \cite{KN06}, where Khot and Naor provide lower bounds for the $L_1$-distortion of quotients of $\ms{C}_n$ by linear codes and by the action of transitive subgroups of the symmetric group $S_n$. As the proofs of \cite{KN06} rely on delicate properties of both these structured quotients and $L_p$ spaces, it seems improbable that they can be easily modified to give nonembeddability results into spaces with given Talagrand type.
\end{remark}


\section{Concluding remarks and open problems} \label{sec:8}

In this final section, we shall present a few remarks regarding the preceeding results and indicate some potentially interesting directions of future research.


\subsection{Talagrand type and linear type}  \label{sec:9.1}

In order to highlight the relation of our results with Talagrand's original inequality~\eqref{eq:talagrand}, we decided to state Theorem~\ref{thm:useivv},~\ref{thm:useeldan},~\ref{thm:useivvorlicz} and~\ref{thm:useeldanorlicz} only for spaces of Rademacher or martingale type 2. In the terminology of Definition~\ref{def:talagrand}, one has the following more general results for spaces of Rademacher or martingale type $s$. Here and throughout, we will denote by $\psi_{s,\delta}:[0,\infty)\to[0,\infty)$ a Young function with $\psi_{s,\delta}(t)=t^s\log^{-\delta}(e+t)$ for large enough $t>0$.

\begin{theorem} [Rademacher type and Talagrand type] \label{thm:gen-tal1}
Fix $s\in(1,2]$. If a Banach space $(E,\|\cdot\|_E)$ has Rademacher type $s$, then for every $\e\in(0,s/2)$, $E$ has Talagrand type $(s,\psi_{s,s/2-\e})$.
\end{theorem}

\begin{theorem} [Martingale type and Talagrand type] \label{thm:gen-tal2}
Fix $s\in(1,2]$. If a Banach space $(E,\|\cdot\|_E)$ has martingale type $s$, then $E$ also has Talagrand type $(s,\psi_{s,s/2})$.
\end{theorem}

Since for every $s\in(1,2]$ there exist spaces of Rademacher type $s$ which do not have martingale type $s$ (see~\cite{Jam78, PX87}), the following natural question poses itself.

\begin{question}
Does every Banach space of Rademacher type $s$ also have Talagrand type $(s,\psi_{s,s/2})$?
\end{question}


\subsection{Talagrand type of $L_1(\mu)$} 

It is worth emphasizing that the proofs of both Theorems~\ref{thm:gen-tal1} and~\ref{thm:gen-tal2} crucially rely on the fact that $s>1$ due to the use of Bonami's hypercontractive inequalities~\cite{Bon70}. In the following theorem, we establish the Talagrand type of $L_1$. It is worth emphasizing the somewhat surprising fact that Theorem~\ref{thm:L1} below shows that a stronger property than the trivial ``Enflo type 1'' inequality holds true in $L_1$.

\begin{theorem} \label{thm:L1}
For every measure $\mu$, the Banach space $L_1(\mu)$ has Talagrand type $(1,\psi_{1,1})$.
\end{theorem}

\begin{proof}
Since Talagrand type is a local invariant, it clearly suffices to consider the case that $\mu$ is the counting measure on $\N$ and thus $L_1(\mu)$ is isometric to $\ell_1$. We will employ a classical result of Schoenberg~\cite{Sch38}, according to which there exists a function $\mathfrak{s}:\R\to\ell_2$ such that $\mathfrak{s}(0)=0$ and
\begin{equation*}
\forall \ a,b\in\R, \qquad \big\|\mathfrak{s}(a)-\mathfrak{s}(b)\big\|_{\ell_2}^2 = |a-b|.
\end{equation*}
Consider the mapping $\overline{\mathfrak{s}}:\ell_1\to\ell_2(\ell_2)$, given by
\begin{equation*}
\overline{\mathfrak{s}}(a_1,a_2,\ldots) = \big(\mathfrak{s}(a_1),\mathfrak{s}(a_2),\ldots\big)
\end{equation*}
and observe that for $a=(a_1,a_2,\ldots), b=(b_1,b_2,\ldots)\in\ell_1$,
\begin{equation*}
\big\|\overline{\mathfrak{s}}(a)-\overline{\mathfrak{s}}(b)\big\|_{\ell_2(\ell_2)}^2 = \sum_{i=1}^\infty \big\|\mathfrak{s}(a_i)-\mathfrak{s}(b_i)\big\|_{\ell_2}^2 = \sum_{i=1}^\infty |a_i-b_i| = \|a-b\|_{\ell_1}.
\end{equation*}
 Fix $n\in\N$ and a function $f:\ms{C}_n\to \ell_1$. Consider the composition $g:\ms{C}_n\to\ell_2(\ell_2)$ given by $g=\overline{\mathfrak{s}}\circ f$. Then, we have
\begin{equation*}
\begin{split}
\mb{E}_{\sigma_n\times\sigma_n} \big\|f(\e)-f(\delta)\big\|_{\ell_1} 
 = \mb{E}_{\sigma_n\times\sigma_n} \big\|g(\e)-g(\delta)\big\|_{\ell_2(\ell_2)}^2 & =  \mb{E}_{\sigma_n}\big\| g(\e)-\mb{E}_{\sigma_n} g\big\|_{\ell_2(\ell_2)}^2 \\ & \lesssim \sum_{i=1}^n \big\|\partial_i g\big\|^2_{L_2(\log L)^{-1}(\sigma_n;\ell_2(\ell_2))},
\end{split}
\end{equation*}
where the last inequality follows from Theorem~\ref{thm:useeldanorlicz}. Combining this with the pointwise identity
\begin{equation*}
\big\|\partial_i g(\e)\big\|_{\ell_2(\ell_2)} = \frac{1}{2}\big\|g(\e)-g(\e_1,\ldots,\e_{i-1},-\e_i,\e_{i+1},\ldots,\e_n)\big\|_{\ell_2(\ell_2)} = \frac{1}{\sqrt{2}} \big\|\partial_if(\e)\big\|^{1/2}_{\ell_1}
\end{equation*}
and the fact that for every $h:\{-1,1\}^n\to\R_+$,
\begin{equation*}
\big\|\sqrt{h}\big\|_{L_2(\log L)^{-1}(\sigma_n)}^2 \asymp \big\|h\big\|_{L_1(\log L)^{-1}(\sigma_n)},
\end{equation*}
we deduce that
\begin{equation*}
\mb{E}_{\sigma_n\times\sigma_n}\big\|f(\e)-f(\delta)\big\|_{\ell_1} \lesssim \sum_{i=1}^n \big\|\partial_if\big\|_{L_1(\log L)^{-1}(\sigma_n;\ell_1)}.
\end{equation*}
This concludes the proof of the theorem.
\end{proof}

The argument used in the proof of Theorem~\ref{thm:L1} to derive the Talagrand type of $\ell_1$ from the Talagrand type of $\ell_2$ is very specifically tailored to $L_1(\mu)$ spaces. It remains an interesting open problem to investigate the Talagrand type of noncommutative $L_1$-spaces.

\begin{question}
Does the Schatten trace class $(\msf{S}_1,\|\cdot\|_{\msf{S}_1})$ have Talagrand type $(1,\psi_{1,1})$?
\end{question}


\subsection{Vector-valued Riesz transforms} 

The optimal $L_1-L_p$ inequality for scalar-valued functions (see Theorem~\ref{thm:l1lporlicz}) was derived by combining the vector-valued Bakry--Meyer inequality of Theorem~\ref{thm:vectorbakrymeyer} and Lust-Piquard's Theorem~\ref{thm:lust-piquard}. In fact, the same argument gives the following implication.

\begin{theorem}
Let $(E,\|\cdot\|_E)$ be a $K$-convex Banach space such that for some $\alpha\in\big(0,\tfrac{1}{2}\big]$, $p\in(1,\infty)$ and $K\in(0,\infty)$, the following property holds. For every $n\in\N$, every $f:\ms{C}_n\to E$ satisfies
\begin{equation} \label{eq:left-riesz}
\big\|\Delta^{\alpha}f\big\|_{L_p(\sigma_n;E)} \leq K \big\| \nabla f\big\|_{L_p(\sigma_n;E)}.
\end{equation}
Then, there exists $C=C(\alpha,p,K)\in(0,\infty)$ such that for every $f:\ms{C}_n\to E$,
\begin{equation*}
\big\|f-\mb{E}_{\sigma_n}f\big\|_{L_p(\sigma_n;E)} \leq C \big\|\nabla f\big\|_{L_p(\log L)^{-\alpha p}(\sigma_n;E)}.
\end{equation*}
\end{theorem}

Therefore, the following question seems natural.

\begin{question}
Fix $\alpha\in\big(0,\tfrac{1}{2}\big]$ and $p\in[1,\infty)$. Which target spaces $(E,\|\cdot\|_E)$ satisfy~\eqref{eq:left-riesz} with a constant $K$ independent of $n$?
\end{question}

In the case of Gauss space, it has been shown by Pisier (see~\cite{Pis88}) that dimension-free Riesz transform inequalities hold true provided that the target space $E$ has the UMD property. In particular, this means that in the case of UMD spaces, Theorem~\ref{thm:lpmporlicz} can be improved as follows.

\begin{theorem} \label{thm:umd}
Let $(E,\|\cdot\|_E)$ be a UMD Banach space. Then, for every $p\in(1,\infty)$, there exists $C_p=C_p(E)\in(0,\infty)$ such that for $n\in\N$, every smooth function $f:\R^n\to E$ satisfies
\begin{equation*}
\big\|f-\mb{E}_{\gamma_n}f\big\|_{L_p(\gamma_n;E)} \leq C_p\cdot \big\|\nabla f\big\|_{L_p(\log L)^{-\frac{p}{2}}(\gamma_n;E)}
\end{equation*}
\end{theorem}


\subsection{Talagrand type of nonpositively curved spaces} 

A geodesic metric space $(\MM,d_\MM)$ is an Alexandrov space of (global) nonpositive curvature (or simply a CAT(0) space) if for every quadruple of points $x,y,z,m\in\MM$ such that $m$ is a metric midpoint of $x$ and $y$, i.e., a point for which $d_\MM(x,m)=d_\MM(y,m) = \tfrac{1}{2}d_\MM(x,y)$, we have
\begin{equation*}
d_\MM (z,m)^2 \leq \frac{1}{2}d_\MM(z,x)^2+\frac{1}{2}d_\MM(z,y)^2 - \frac{1}{4}d_\MM(x,y)^2.
\end{equation*}
Complete Riemannian manifolds of nonpositive sectional curvature are examples of CAT(0) spaces. Let \mbox{$\psi_{2,\delta}:[0,\infty)\to[0,\infty)$} be a Young function with $\psi_{2,\delta}(t) = t^2 \log^{-\delta}(e+t)$ for large enough $t>0$. In Theorems~\ref{thm:gromov} and~\ref{thm:pinched}, we showed that Gromov hyperbolic groups and complete Riemannian manifolds of pinched negative curvature have Talagrand type $(2,\psi_{1-\e})$ for every $\e\in(0,1)$. On the other hand, a classical inductive argument essentially going back to Enflo~\cite{Enf69} shows that all Alexandrov spaces of nonpositive curvature have Enflo type 2, which is closely related to Talagrand type $(2,\psi_{2,0})$. We believe that the following question deserves further investigation.

\begin{question}
Does there exist some $\delta\in(0,1]$ such that every Alexandrov space of nonpositive curvature has Talagrand type $(2,\psi_{2,\delta})$? More ambitiously, does every Alexandrov space of nonpositive curvature have Talagrand type $(2,\psi_{2,1})$?
\end{question}




\subsection{CAT(0) spaces as test spaces for superreflexivity}

 In Proposition~\ref{prop:bourgain}, we showed that all binary $\R$-trees of finite depth embed with uniformly bounded distortion into any nonsuperreflexive Banach space. It was communicated to us by Florent Baudier that using this proposition and the barycentric gluing technique (see~\cite{Bau07} and the survey~\cite{Bau14}), one can in fact prove that the binary $\R$-tree of {\it infinite} depth admits a bi-Lipschitz embedding into any nonsuperreflexive Banach space. Then, an inductive argument (see, e.g.,~\cite[Remark~2.2]{Ost14}) shows that any finite product of binary $\R$-trees also embeds bi-Lipschitzly into any nonsuperreflexive Banach space. Therefore, one deduces from Theorem~\ref{thm:npss} that every finite-dimensional complete simply connected Riemannian manifold of pinched negative curvature embeds bi-Lipschitzly into any nonsuperreflexive Banach space. Conversely, since all binary trees embed in the hyperbolic plane $\mb{H}^2$, if a Banach space $E$ bi-Lipschitzly contains $\mb{H}^2$, then $E$ cannot be superreflexive by Bourgain's theorem~\cite{Bou86}. In conclusion, we deduce the following characterization. 

\begin{theorem}
A Banach space $(E,\|\cdot\|_E)$ is nonsuperreflexive if and only if for every $n\in\N$, every $n$-dimensional complete, simply connected Riemannian manifold $(\msf{M},g)$ of pinched negative curvature equipped with the Riemannian distance $d_\msf{M}$ admits a bi-Lipschitz embedding in $E$.
\end{theorem}

In recent years, there have been plenty of such characterizations in the literature, although one can argue that this is not a particularly novel one due to its close relation to Bourgain's characterization in terms of trees. We believe the following stronger \mbox{question deserves further investigation.}

\begin{question} \label{q:nonsuper}
Which Alexandrov spaces of nonpositive curvature admit a bi-Lipschitz embedding into every nonsuperreflexive Banach space?
\end{question}

There are plenty of CAT(0) spaces which do not embed into finite products of binary $\R$-trees and in order to prove that they embed into all nonsuperreflexive Banach spaces, one may need to employ interesting structural properties of such spaces. On the other hand, there exist CAT(0) spaces which do not embed into $L_1$, which is of course nonsuperreflexive. Indeed, if every CAT(0) space admitted a bi-Lipschitz embedding into $L_1$, then every classical expander (which is also an expander with respect to $L_1$ by Matou\v{s}ek's extrapolation lemma for Poincar\'e inequalities, see~\cite{Mat97}), would be an expander with respect to all CAT(0) spaces and this is known to be false by important work of Mendel and Naor~\cite{MN15}.


\subsection{General hypercontractive semigroups}

 In~\cite{CL12}, Cordero-Erausquin and Ledoux established versions of Talagrand's (scalar-valued) inequality~\eqref{eq:talagrand} in the setting of hypercontractive Markov semigroups satisfying some minimal assumptions. At first glance, the arguments which we use in the present paper to obtain vector-valued extensions of~\eqref{eq:talagrand} seem to rely more heavily in specific properties of the Hamming cube, such as identity~\eqref{eq:crucial-ivv} from~\cite{IVV20} or the Eldan--Gross process~\cite{EG19}. Nevertheless, we strongly believe that there are versions of our results for other hypercontractive Markov semigroups satisfying some fairly general assumptions.


\bibliographystyle{alpha}
\bibliography{L1L2}

\end{document}